\documentclass{amsart}
\pdfoutput=1
\usepackage[margin=1.2in]{geometry} 
\usepackage[utf8]{inputenc}
\usepackage[T1]{fontenc}
\usepackage{amsmath, amsthm, amssymb, cancel, enumerate, mathrsfs, mathtools}
\usepackage{setspace}                   
\usepackage{indentfirst}                
\usepackage{type1cm}                    
\usepackage{titletoc}
\usepackage{dsfont} 
\usepackage{physics}
\usepackage{subcaption}
\usepackage{enumerate}
\usepackage{fancybox}
\usepackage{wasysym}
\usepackage{comment}
\usepackage{longtable}
\usepackage{bm}
\usepackage{upgreek}
\usepackage{tikz}
\usepackage[square,sort,comma,numbers]{natbib}
\usepackage{ textcomp }
\usepackage{url}

\graphicspath{{./figuras/}}             
\fontsize{60}{62}\usefont{OT1}{cmr}{m}{n}{\selectfont}
\cleardoublepage
\normalsize

\theoremstyle{definition}
\newtheorem{Def}{\textsc{Definition}}[section]

\newtheorem{Rem}{\textsc{Remark}}

\theoremstyle{plain}

\newtheorem{Ques}[Def]{\textsc{Question}}
\newtheorem{Lema}[Def]{\textsc{Lemma}}
\newtheorem{Corol}[Def]{\textsc{Corollary}}
\newtheorem{Teo}[Def]{\textsc{Theorem}}
\newtheorem{Cla}{Claim}

\newcommand{\restr}[2]{\left. #1 \right|_{#2}}

\makeatletter
\def\moverlay{\mathpalette\mov@rlay}
\def\mov@rlay#1#2{\leavevmode\vtop{%
   \baselineskip\z@skip \lineskiplimit-\maxdimen
   \ialign{\hfil$\m@th#1##$\hfil\cr#2\crcr}}}
\newcommand{\charfusion}[3][\mathord]{
    #1{\ifx#1\mathop\vphantom{#2}\fi
        \mathpalette\mov@rlay{#2\cr#3}
      }
    \ifx#1\mathop\expandafter\displaylimits\fi}
\makeatother

\newcommand{\cupdot}{\charfusion[\mathbin]{\cup}{\cdot}}
\newcommand{\bigcupdot}{\charfusion[\mathop]{\bigcup}{\cdot}}

\usepackage{scalerel,stackengine}
\stackMath
\newcommand\reallywidehat[1]{%
\savestack{\tmpbox}{\stretchto{%
  \scaleto{%
    \scalerel*[\widthof{\ensuremath{#1}}]{\kern-.6pt\bigwedge\kern-.6pt}%
    {\rule[-\textheight/2]{1ex}{\textheight}}
  }{\textheight}%
}{0.5ex}}%
\stackon[1pt]{#1}{\tmpbox}%
}
\usepackage[utf8]{inputenc}

\title{On powers of countably pracompact groups}
\author{Artur Hideyuki Tomita$^{\dag}$}
\thanks{$^{\dag}$ The first listed author has received financial support from \textit{São Paulo Research Foundation} (FAPESP), grant 2021/00177-4}
\author{Juliane Trianon-Fraga$^{\ddag}$}
\thanks{$^{\ddag}$ The second listed author has received financial support from \textit{São Paulo Research Foundation} (FAPESP), grant 2019/12628-0.}
\newcommand{\Addresses}{
  \bigskip
  
\indent \textsc{Depto de Matemática, Instituto de Matemática e Estatística, Universidade de São Paulo, Rua do Matão, 1010, CEP 05508-090, São Paulo, SP, Brazil.}\par\nopagebreak
  \textit{E-mail addresses:} tomita@ime.usp.br (A. H. Tomita), jtrianon@ime.usp.br (J. Trianon-Fraga).}

\begin{document}
\begin{abstract}
In 1990, Comfort asked: is there, for every cardinal number $\alpha \leq 2^{\mathfrak{c}}$, a topological group $G$ such that $G^\gamma$ is countably compact for all cardinals $\gamma<\alpha$, but $G^\alpha$ is not countably compact? A similar question can also be asked for countably pracompact groups: for which cardinals $\alpha$ is there a topological group $G$ such that $G^{\gamma}$ is countably pracompact for all cardinals $\gamma < \alpha$, but $G^{\alpha}$ is not countably pracompact? In this paper we construct such group in the case $\alpha = \omega$, assuming the existence of $\mathfrak{c}$ incomparable selective ultrafilters, and in the case $\alpha = \kappa^{+}$, with $\omega \leq \kappa \leq 2^{\mathfrak{c}}$, assuming the existence of $2^{\mathfrak{c}}$ incomparable selective ultrafilters. In particular, under the second assumption, there exists a topological group $G$ so that $G^{2^\mathfrak{c}}$ is countably pracompact, but $G^{{(2^{\mathfrak{c}})}^+}$ is not countably pracompact, unlike the countably compact case.  


\end{abstract}

\maketitle

\section{Introduction}

Throughout this paper, every topological space will be Tychonoff (Hausdorff and completely regular) and every topological group will be Hausdorff (thus, also Tychonoff). For an infinite set $X$, $[X]^{<\omega}$ will denote the family of all finite subsets of $X$, and $[X]^{\omega}$ will denote the family of all countable subsets of $X$. Recall that an infinite topological space $X$ is said to be
\begin{itemize}
    \item \textit{pseudocompact} if each continuous real-valued function on $X$ is bounded;
    \item \textit{countably compact} if every infinite subset of $X$ has an accumulation point in $X$;
    \item \textit{countably pracompact} if there exists a dense subset $D$ in $X$ such that every infinite subset of $D$ has an accumulation point in $X$.
\end{itemize}

We denote the set of non-principal (free) ultrafilters on $\omega$ by $\omega^*$. The following notion was introduced by Bernstein \cite{Bernstein}:

\begin{Def}[\cite{Bernstein}]
Let $p \in \omega^*$ and $\{x_n: n \in \omega\}$ be a sequence in a topological space $X$. We say that $x \in X$ is a $p-$\textit{limit point} of $\{x_n:n \in \omega\}$ if $\{n \in \omega: x_n \in U\} \in p$ for every neighborhood $U$ of $x$.
\end{Def}
Notice that if $X$ is a Hausdorff space, for each $p \in \omega^*$ a sequence $\{x_n: n \in \omega\} \subset X$ has at most one $p$-limit point $x$ and we write $x=p-\lim_{n \in \omega}x_n$ in this case.

One may write the compact-like definitions above using the notion of $p-$limits. In fact, it is not hard to show that $x \in X$ is an accumulation point of a sequence $\{x_n: n \in \omega\} \subset X$ if and only if there exists $p \in \omega^*$ such that $x= p-\lim_{n \in \omega}x_n$. Thus, we have that 
\begin{itemize}
\item $X$ is countably compact if and only if every sequence $\{x_n:n \in \omega\} \subset X$ has a $p-$limit, for some $p \in \omega^*$.
    \item $X$ is countably pracompact if and only if there exists a dense subset $D$ in $X$ such that every sequence $\{x_n: n \in \omega\} \subset D$ has a $p-$limit in $X$, for some $p \in \omega^*$.
\end{itemize} 
For pseudocompact spaces, a similar equivalence holds: $X$ is pseudocompact if and only if for every countable family $\{U_n: n \in \omega\}$ of nonempty open sets of $X$, there exists $x \in X$ and $p \in \omega^*$ such that, for each neighborhood $V$ of $x$, $\{n \in \omega: V \cap U_n \neq \emptyset\} \in p$.
 
There are many concepts related to compactness and pseudocompactness which have emerged in the last years. In this paper, we highlight the following, which was introduced in \cite{yasser2}.

\begin{Def}[\cite{yasser2}]
A topological space $X$ is called \textit{selectively pseudocompact}\footnotemark \ if for each sequence $\{U_n: n \in \omega\}$ of nonempty open subsets of $X$ there is a sequence $\{x_n: n \in \omega\} \subset X$, $x \in X$ and $p \in \omega^*$ such that $x=p-\lim_{n \in \omega}x_n$ and, for each $n \in \omega$, $x_n \in U_n$.
\footnotetext{This concept was originally defined under the name \textit{strong pseudocompactness}, but later the name was changed, since there were already two different properties named in the previous way (in \cite{propriedade1} and \cite{propriedade2}).}
\end{Def}

It is clear that every selectively pseudocompact space is pseudocompact and every countably pracompact space is selectively pseudocompact. Also, it was proved in \cite{tomita1} that there exists a pseudocompact topological group which is not selectively pseudocompact, and in \cite{trianontomita2022} that there exists a selectively pseudocompact group which is not countably pracompact.
\vspace{0.5cm}

We shall now briefly recall the definitions and some facts about selective ultrafilters and the Rudin-Keisler order.

\begin{Def}
A \textit{selective ultrafilter} on $\omega$ is a free ultrafilter $p$ on $\omega$ such that for every partition $\{A_n: n \in \omega\}$ of $\omega$, either there exists $n \in \omega$ such that $A_n \in p$ or there exists $B \in p$ such that $|B \cap A_n| =1$ for every $n \in \omega$.
\end{Def}

Given an ultrafilter $p$ on $\omega$ and a function $f: \omega \to \omega$, note that 
\[
f_{*}(p) \doteq \{A \subset \omega: f^{-1}(A) \in p\}
\]
is also an ultrafilter on $\omega$. Consider then the following definition.
\begin{Def}
Given $p, q \in \omega^{*}$, we say that $p \leq_{RK} q$ if there exists a function $f: \omega \to \omega$ so that $f_{*}(q)=p$. Such relation on $\omega^*$ is a preorder called the \textit{Rudin-Keisler order}.  
\end{Def}

We say that $p, q \in \omega^*$ are:
\begin{itemize}
\item \textit{incomparable} if neither $p \leq_{RK} q$ or $q \leq_{RK} p$;
\item \textit{equivalent} if $p \leq_{RK} q$ and $q \leq_{RK} p$.
\end{itemize}

The existence of selective ultrafilters is independent of ZFC. In fact, there exists a model of ZFC in which there are no $P-$points\footnotemark \ in $\omega^*$ \cite{Shelah}, while Martin's axiom (MA) implies the existence of $2^{\mathfrak{c}}$ incomparable selective ultrafilters \cite{Blass}.
\footnotetext{A free ultrafilter $p \in \omega^*$ is a $P$-\textit{point} if, for every sequence $(A_n)_{n \in \omega}$ of elements of $p$, there exists $A \in p$ so that $A \setminus A_n$ is finite for each $n \in \omega$. Every selective ultrafilter is a $P-$point.}
\vspace{0.5cm} 

Pseudocompactness is not preserved under products for arbitrary topological spaces \cite{terasaka}, but interestingly Comfort and Ross proved that the product of any family of pseudocompact topological groups is pseudocompact \cite{comross}. This result motivated Comfort to question whether the product of countably compact groups is also countably compact. More generally, he asked the following question \cite{QuestionComfort}:
\begin{Ques}[\cite{QuestionComfort}, Question 477]\label{comquest}
Is there, for every (not necessarily infinite) cardinal number $\alpha  \leq 2^{\mathfrak{c}}$, a topological group $G$ such that $G^{\gamma}$ is countably compact for all cardinals $\gamma < \alpha$, but $G^{\alpha}$ is
not countably compact?
\end{Ques}

The restriction $\alpha \leq 2^{\mathfrak{c}}$ in the question above is due to the following result:

\begin{Teo}[\cite{ginsa}, Theorem 2.6]
Let $X$ be a Hausdorff topological space. The following statements are equivalent:
\begin{enumerate}[(i)]
\item every power of $X$ is countably compact;
\item $X^{2^\mathfrak{c}}$ is countably compact;
\item $X^{|X|^{\omega}}$ is countably compact;
\item there exists $p \in \omega^*$ such that $X$ is $p-$compact\footnotemark.
\end{enumerate}
\end{Teo}

\footnotetext{Given $p \in \omega^*$, a topological space $X$ is \textit{p-compact} if every sequence of points in $X$ has a $p-$limit. The product of $p-$compact spaces is $p-$compact, for every $p \in \omega^*$.}

Van Douwen was the first to prove consistently (under MA) that there are two countably compact groups whose product is not countably compact \cite{Douwen}. Also, Question \ref{comquest} was answered positively in \cite{tomita34}, assuming the existence of $2^{\mathfrak{c}}$ selective ultrafilters and that $2^{\mathfrak{c}}=2^{<2^{\mathfrak{c}}}$. Finally, in 2021, it was proved in ZFC that there are two countably compact groups whose product is not countably compact \cite{Michael}.

It is natural also to ask productivity questions for countably pracompact and selectively pseudocompact groups. In this regard, Garcia-Ferreira and Tomita proved that if $p$ and $q$ are non-equivalent (according to the Rudin-Keisler order in $\omega^*$) selective ultrafilters on $\omega$, then there are a $p$-compact group and a $q$-compact group whose product is not selectively pseudocompact \cite{tomita2}. Also, Bardyla, Ravsky and Zdomskyy constructed, under MA, a Boolean countably compact topological
group whose square is not countably pracompact \cite{Bardyla}. However, the following questions remain unsolved in ZFC.
\begin{Ques}[ZFC]
Is it true that selective pseudocompactness is non-productive in the class of topological groups?
\end{Ques}
\begin{Ques}[ZFC]
Is it true that countable pracompactness is non-productive in the class of topological groups?
\end{Ques}

More generally, one can ask Comfort-like questions, such as Question \ref{comquest}, for selectively pseudocompact and countably pracompact groups. In the case of selectively pseudocompact groups, the question is restricted to cardinals $\alpha \leq \omega$, due to the next result. 
\begin{Lema}
If $G$ is a topological group such that $G^{\omega}$ is selectively pseudocompact, then $G^{\kappa}$ is selectively pseudocompact for every cardinal $\kappa \geq \omega$.
\end{Lema}
\begin{proof}
Indeed, let $\kappa \geq \omega$ and $(U_n)_{n \in \omega}$ be a family of open subsets of $G^{\kappa}$. For every $n \in \omega$, there are open subsets $U_n^j \subset G$, for each $j< \kappa$, so that $\prod_{j \in \kappa} U_n^j \subset U_n$ and $U_n^j \neq G$ if and only if $j \in F_n$, for a finite subset $F_n \subset \kappa$. Let $F \doteq \bigcup_{n \in \omega}F_n$. For each $n \in \omega$, consider the open subsets $V_n \doteq \prod_{j \in F_n} U_n^j \times \prod_{j \in F \setminus F_n} G \subset G^{F}$. By assumption, $G^F$ is selectively pseudocompact, thus there is a sequence $\{y_n: n \in \omega\} \subset G^F$ so that $y_n \in V_n$, for every $n \in \omega$, which has an accumulation point $y$ in $G^F$. Then, given $g \in G$ arbitrarily,
the sequence $\{x_n: n \in \omega\} \subset G^{\kappa}$ defined coordinatewise, for each $n \in \omega$, by
\[
    x_n^j \doteq
    \begin{cases}
    y_n^j, & \text{if} \ j \in F\\
    g, & \text{if } j \in \kappa \setminus F
\end{cases}
\]
is such that $x_n \in U_n$ for every $n \in \omega$, and has $x \in G^{\kappa}$ given by
\[
    x^j \doteq
    \begin{cases}
    y^j, & \text{if} \ j \in F\\
    g, & \text{if } j \in \kappa \setminus F
\end{cases}
\]
as accumulation point. 
\end{proof}

\begin{Ques}\label{comsel}
For which cardinals $\alpha \leq \omega$ is there a topological group $G$ such that $G^{\gamma}$ is selectively pseudocompact for all cardinals $\gamma < \alpha$, but $G^{\alpha}$ is not selectively pseudocompact?
\end{Ques}

In the case of countably pracompact groups, it is still not known whether there exists a cardinal $\kappa$ satisfying that: if a topological group $G$ is such that $G^\kappa$ countably pracompact, then $G^\alpha$ is countably pracompact, for each $\alpha> \kappa$. Thus, there is no restriction to the cardinals $\alpha$ yet:

\begin{Ques}\label{compra}
For which cardinals $\alpha$ is there a topological group $G$ such that $G^{\gamma}$ is countably pracompact for all cardinals $\gamma < \alpha$, but $G^{\alpha}$ is not countably pracompact?
\end{Ques}
It is worth observing that if $G^{\omega}$ is \textbf{countably compact} and $\kappa \geq \omega$, then 
\[
\Sigma \doteq \{g \in G^{\kappa}: |\{\alpha \in \kappa: g^{\alpha} \neq 0\}| \leq \omega \}
\]
is a dense subset of $G^{\kappa}$ for which every infinite subset has an accumulation point. Thus, in this case $G^{\kappa}$ is countably pracompact.

In \cite{finpow}, under the assumption of CH, the authors showed that for every positive integer $k>0$, there exists a topological group $G$ for which $G^k$ is countably compact but $G^{k+1}$ is not selectively pseudocompact. Thus, Question \ref{comsel} and Question \ref{compra} are already solved for finite cardinals under CH. The cardinal $\alpha = \omega$ is the only one for which there are still no consistent answers to the Question \ref{comsel}. 

In this paper:
\begin{itemize}
    \item assuming the existence of $\mathfrak{c}$ incomparable selective ultrafilters, we answer Question \ref{compra} for $\alpha = \omega$;
    \item assuming the existence of $2^\mathfrak{c}$ incomparable selective ultrafilters, we answer Question \ref{compra} for each successor cardinal $\alpha = \kappa^+$, with $\omega \leq \kappa \leq 2^\mathfrak{c}$.
\end{itemize}
 
We will be dealing with Boolean groups, which are also vector spaces over the field $2=\{0,1\}$, and thus we can talk about general linear algebra concepts concerning these groups, such as \textit{linearly independent subsets}. More specifically, if $D \subset 2^{\mathfrak{c}}$ is an infinite set, we will consider $[D]^{<\omega}$ as a Boolean group, with the symmetric difference $\triangle$ as the group operation and $\emptyset$ as the neutral element.
 
 Given $p \in \omega^*$, one may define an equivalence relation on $([D]^{<\omega})^{\omega}$ by letting $f \equiv_p g$ iff $\{n \in \omega: f(n)=g(n)\} \in p$. We let $[f]_p$ be the equivalence class determined by $f$ and $([D]^{<\omega})^{\omega}/p$ be $([D]^{<\omega})^{\omega}/\equiv_p$. Notice that this set has a natural vector space structure (over the field $2$). For each $D_0 \in [D]^{<\omega}$, the constant function in $([D]^{<\omega})^{\omega}$ which takes only the value $D_0$ will be denoted by $\vec{D_0}$.

 \section{Auxiliary results}
 In this section we present the auxiliary results that we will use in the constructions. We start with a simple linear algebra lemma, stated and proved in \cite{trianontomita2022}.
 
  \begin{Lema}[\cite{trianontomita2022}]\label{novo}
    Let $A$, $B$ and $C$ be subsets in a Boolean group. Suppose that $A$ is a finite set and that $A \cup C$, $B \cup C$ are linearly independent. Then there exists $B' \subset B$ such that $|B'| \leq |A|$ and $A \cup C \cup (B \setminus B')$ is linearly independent.
    \end{Lema}
 
 The next two technical lemmas will also be useful.
 
 \begin{Lema}\label{caso2teo2}
Let $X$ be an infinite set and $\{X_0,...,X_n\}$ be a partition of $X$. Let also $(x_k)_{k \in \omega}$ and $(y_k)_{k \in \omega}$ be sequences in the Boolean group $[X]^{<\omega}$ so that:
\begin{itemize}
    \item $\{x_k: k \in \omega\} \cup \{y_k: k \in \omega\}$ is linearly independent;
    \item for every $p \in \{0,...,n\}$, both $\{x_k \cap X_p: k \in \omega\}$ and $\{y_k \cap X_p: k \in \omega\}$ are linearly independent.
\end{itemize}

Then, there exist a subsequence $\{k_m: m \in \omega\}$ and $n_0 \in \{0,...,n\}$ so that 
\[\{x_{k_m} \cap X_{n_0}: m \in \omega\} \cup \{y_{k_m} \cap X_{n_0}: m \in \omega\}\]
is linearly independent.
\end{Lema}
\begin{proof}
We shall construct inductively a sequence $(A_0^i)_{i \in \omega}$ of subsets of $\omega$ as follows.  Firstly, if does not exist $k \in \omega$ so that $\{x_k \cap X_0\} \cup \{y_k \cap X_0\}$ is linearly independent, we put $A_0^0 = \emptyset$. Otherwise, we choose the minimum $k_0 \in \omega$ with this property and put $A_0^0 \doteq \{k_0\}$. Suppose that for $l \in \omega$ we have constructed $A_0^0,...,A_0^l \subset \omega$ such that:
\begin{enumerate}[i)]
    \item $|A_0^i| \leq i+1$, for each $i=0,...,l$;
    \item $A_0^i \subset A_0^j$ if $0\leq i \leq j \leq l$;
    \item $\{x_k \cap X_0: k \in A_0^l\} \cup \{y_k \cap X_0: k \in A_0^l\}$ is linearly independent.
    \item for each $0 \leq i < l$, $A_0^{i+1} \setminus A_0^{i} = \emptyset$ if, and only if, 
    \[\{x_k \cap X_0: k \in A_0^i\} \cup \{y_k \cap X_0: k \in A_0^i\} \cup \{x_{\tilde{k}} \cap X_{0}\} \cup \{y_{\tilde{k}} \cap X_{0}\}\]
    is linearly dependent for every $\tilde{k} > \text{max}(A_0^i)$.
\end{enumerate}
In what follows, we will construct $A_0^{l+1}$. If does not exist $\tilde{k} \in \omega$, $\tilde{k} > \text{max}(A_0^l)$, so that 
\[
\{x_k \cap X_0: k \in A_0^l\} \cup \{y_k \cap X_0: k \in A_0^l\} \cup \{x_{\tilde{k}} \cap X_{0}\} \cup \{y_{\tilde{k}} \cap X_{0}\}
\]
is linearly independent, we put $A_0^{l+1} = A_0^l$. Otherwise, we choose the minimum $k_{l+1} \in \omega$ with this property, and put $A_0^{l+1}= A_0^l \cup \{k_{l+1}\}$. In any case, $A_0^0,...,A_0^{l+1}$ satisfy items i)--iv), and then, by induction, there exists a sequence $(A_0^i)_{i \in \omega}$ satisfying them. Now, let $A_0 \doteq \bigcup_{i \in \omega} A_0^i$. If $A_0$ is infinite, then $\{x_k \cap X_0: k \in A_0\} \cup \{x_k \cap X_0: k \in A_0\}$ is linearly independent, and we are done. 

On the other hand, suppose that $A_0$ is finite. We may repeat the process above for $X_1,...,X_n$, constructing analogous subsets $A_1,...,A_n \subset \omega$. If either of them is infinite, we are done.

Suppose then that $A_0,...,A_n$ are finite sets. By construction, for each $\tilde{k} > \text{max}(A_0 \cup... \cup A_n)$ and $j=0,...,n$,
\[
\{x_k \cap X_j: k \in A_j\} \cup \{y_k \cap X_j: k \in A_j\} \cup \{x_{\tilde{k}} \cap X_j\} \cup \{y_{\tilde{k}} \cap X_j\}
\]
is linearly dependent. Also, since, for every $j=0,...,n$, 
\[\mathcal{C}_j \doteq \text{span}(\{x_k \cap X_j: k \in A_j\} \cup \{y_k \cap X_j: k \in A_j\})\]
is finite  and both $\{x_k \cap X_j: k \in \omega\}$ and $\{y_k \cap X_j: k \in \omega\}$ are linearly independent, we can fix:
\begin{itemize}
    \item an infinite subset $A \subset \omega$;
    \item $c_j \in \mathcal{C}_j$, for each $j=0,...,n$,
\end{itemize}
so that
\[
x_{\tilde{k}} \cap X_j = (y_{\tilde{k}} \cap X_j) \triangle c_j,
\]
for every $\tilde{k} \in A$ and $j =0,...,n$. Thus,
\[
x_{\tilde{k}} = (x_{\tilde{k}} \cap X_0) \triangle ... \triangle (x_{\tilde{k}} \cap X_n) = (y_{\tilde{k}} \cap X_0) \triangle ... \triangle (y_{\tilde{k}} \cap X_n) \triangle (c_0 \triangle ... \triangle c_n)= y_{\tilde{k}} \triangle (c_0 \triangle ... \triangle c_n),
\]
for every $\tilde{k} \in A$, which is a contradiction, as $\{x_k: k \in \omega\} \cup \{y_k: k \in \omega\}$ is linearly independent. Hence, $A_0,...,A_n$ cannot all be finite.
\end{proof}

 \begin{Lema}\label{bastali}
  Let $X$ be an infinite set, $k >0$ and $\{(x^0_n,...,x^{k-1}_n): n \in \omega\} \subset ([X]^{<\omega})^{k}$ be a sequence. Then, there are:
  \begin{itemize}
      \item elements $d_0,...,d_{k-1} \in [X]^{<\omega}$;
      \item a subsequence $\{(x^0_{n_l},...,x^{k-1}_{n_l}) : l \in \omega\}$;
      \item for some\footnotemark \ $0 \leq t \leq k$, a sequence $\{(y^0_{n_l},...,y_{n_l}^{t-1}): l \in \omega\} \subset ([X]^{<\omega})^t$;
      \item for each $0 \leq s<k$, a function $P_s: t \to 2$,
  \end{itemize} \footnotetext{If $t =0$, we understand that there is no such sequence and item i) becomes: $x_{n_l}^s = d_s$, for every $l \in \omega$ and $0 \leq s <k$.}
satisfying that
  \begin{enumerate}[i)]
      \item $x_{n_l}^s = \Big( \displaystyle \sum_{i=0}^{t-1}P_s(i)y_{n_l}^i \Big) \triangle d_s,$
      for every $l \in \omega$ and $0 \leq s<k$;
      \item $\{y_{n_l}^i: l \in \omega, 0 \leq i<t\}$ is linearly independent.
  \end{enumerate}
 \end{Lema}
 \begin{proof}
Fix $q \in \omega^*$, and let
\[
\mathcal{M} \doteq \Big\{c \in [X]^{<\omega}: [\vec{c}]_q \in \text{span}(\{[x^0]_q,...,[x^{k-1}]_q\}) \Big\}.
\]
It is clear that $\mathcal{M}$ is a finite set, thus let $j \geq 0$ and $\{c^0,...,c^{j-1}\} \subset \mathcal{M}$ be so that $\{c^0,...,c^{j-1}\}$ is a basis for span$(\mathcal{M}) \subset [X]^{<\omega}$. Then, let also $t \geq 0$ and $y^0,...,y^{t-1} \in ([X]^{<\omega})^{\omega}$ be so that $\mathcal{B} \doteq \{[\vec{c^0}]_q,...,[\vec{c^{j-1}}]_q,[y^0]_q,...,[y^{t-1}]_q\}$ is a basis for span$(\{[x^0]_q,...,[x^{k-1}]_q\})$. Hence, there are $A \in q$, $P_s:t \to 2$ and $C_s:j \to 2$, for each $0 \leq s<k$, so that
\[
x_n^s = \sum_{i=0}^{t-1}P_s(i)y_n^i \triangle \sum_{i=0}^{j-1}C_s(i)c^i,
\]
for every $n \in A$ and $0 \leq s <k$. For each $0 \leq s<k$, let $d_s \doteq \sum_{i=0}^{j-1}C_s(i)c^i$.

We shall prove that there exists an infinite subset $I \subset A$ so that $\{y_n^i: n \in I, 0 \leq i<t\}$ is linearly independent. First, note that for each $c \in [X]^{<\omega}$ and nontrivial function $P:t \to 2$ we have that
\[
\sum_{i=0}^{t-1} P(i) [y^i]_q \neq [\vec{c}]_q.
\]
Therefore, there exist a subset $A_{P,c} \subset A$, $A_{P,c} \in q$, so that 
\[
\sum_{i=0}^{t-1} P(i)y_n^i \neq c
\]
for each $n \in A_{P, c}$. In particular, we conclude that $\{y_n^i: 0 \leq i < t-1\}$ is linearly independent for every $n \in \displaystyle \bigcap_{\substack{P:t \to 2\\ P \neq 0}}A_{P,\emptyset} \doteq A_0$. We may choose $n_0 \in A_0$.

Now, suppose that, given $p \geq 1$, for each $l=0,...,p-1$ we have constructed $A_l \in q$ and $n_l \in A_l$ so that $\{y_{n_l}^i: 0 \leq l<p, 0 \leq i<t\}$ is linearly independent, $(n_l)_{0 \leq l <p}$ is strictly increasing and $A_l \subset A$. Let $\mathcal{C}_p \doteq \text{span}(\{y_{n_l}^i: 0 \leq l<p, 0\leq i< t\})$, 
\[
A_p \doteq \bigcap_{\substack{c \in \mathcal{C}_p\\ P:t \to 2\\ P \neq 0}} A_{P,c}  \ (\subset A),
\]
and fix $n_p \in A_p$, $n_p > n_{p-1}$. It is clear that $A_p \in q$ and also $\{y_{n_l}^i: 0 \leq l\leq p, 0 \leq i<t\}$ is linearly independent, by construction. Then, by induction, there are a sequence $(A_l)_{l \in \omega}$ of elements of $q$ and a strictly increasing sequence $(n_l)_{l \in \omega}$ of naturals so that $n_l \in A_l$ and $\{y_{n_l}^{i}: l \in \omega, 0 \leq i <t\}$ is linearly independent. Furthermore,
\[
x_{n_l}^s = \Big( \sum_{i=0}^{t-1}P_s(i)y_{n_l}^i \Big) \triangle d_s,
\]
for every $l \in \omega$ and $0 \leq s<k$. 
\end{proof}

 Next, we enunciate \textbf{Lemma 3.5} and \textbf{Lemma 3.6} of \cite{tomita34}, and an immediate consequence of \textbf{Lemma 2.1} of \cite{tomita1}.

 \begin{Lema}[\cite{tomita34}, Lemma 3.5]\label{lema3.5}
Let $p_0$ and $p_1$ be incomparable selective ultrafilters. Let $\{a_k^j: k \in \omega\} \in p_j$ be a strictly increasing sequence such that $a_k^j>k$ for every $k \in \omega$ and $j \in 2$. Then there exist subsets $I_0$ and $I_1$ of $\omega$ such that:
\begin{enumerate}[(i)]
\item $\{a_k^j: k \in I_j\} \in p_j$ for each $j \in 2$;
\item $\{[k,a_k^j]:j \in 2, k \in I_j\}$ are pairwise disjoint intervals of $\omega$.
\end{enumerate}
 \end{Lema}

 As a corollary of the previous lemma, we obtain:

\begin{Lema}\label{finitolincom}
Let $n>0$ and $\{p_j: j \leq n\}$ be incomparable selective ultrafilters. Let $\{a_k^j: k \in \omega\} \in p_j$ be a strictly increasing sequence such that $a_k^j>k$ for every $k \in \omega$ and $j \leq n$. Then there exists a family $\{I_j: j \leq n\}$ of subsets of $\omega$ such that:
\begin{enumerate}[(i)]
\item $\{a_k^j: k \in I_j\} \in p_j$ for each $j \leq n$;
\item $\{[k,a_k^j]:j \leq n , k \in I_j\}$ are pairwise disjoint intervals of $\omega$.
\end{enumerate}
 \end{Lema}
 \begin{proof} We will show that the lemma is true for each $n >0$ by induction. The case $n=1$ is just Lemma \ref{lema3.5}. 
 
 Suppose that the result is true for a given $n_0 >0$. We claim that it is also true for $n_0+1$. Indeed, let $\{p_j: j \leq n_0+1\}$ be incomparable selective ultrafilters and $\{a_k^j: k \in \omega\} \in p_j$ be a strictly increasing sequence such that $a_k^j>k$ for every $k \in \omega$ and $j \leq n_0+1$. By hypothesis, there exists a family $\{\tilde{I_j}: j \leq n_0\}$ of subsets of $\omega$ so that:
 \begin{itemize}
\item $\{a_k^j: k \in \tilde{I_j}\} \in p_j$ for each $j \leq n_0$;
\item $\{[k,a_k^j]:j \leq n_0 , k \in \tilde{I_j}\}$ are pairwise disjoint intervals of $\omega$.
\end{itemize}

Also, by Lemma \ref{lema3.5}, for each $j \leq n_0$ there exist $I_j \subset \tilde{I_j}$ and $K_j \subset \omega$ so that:
  \begin{itemize}
\item $\{a_k^j: k \in I_j\} \in p_j$ and $\{a_k^{n_0+1}: k \in K_j\} \in p_{n_0+1}$;
\item $\{[k,a_k^j]: k \in I_j\} \cup \{[k,a_k^{n_0+1}]: k \in K_j\}$ are pairwise disjoint intervals of $\omega$.
\end{itemize}
 Then, defining $I_{n_0+1} \doteq \bigcap_{j=0}^{n_0}K_j$, we have that $\{I_j:j \leq n_0+1\}$ satisfies the hypothesis we want. Therefore, the lemma is true for every $n >0$.
 \end{proof}
 
The countable version of the previous result is \textbf{Lemma 3.6} of \cite{tomita34}:

 \begin{Lema}[\cite{tomita34}, Lemma 3.6]\label{enumeravelincom}
 Let $\{p_j: j \in \omega\}$ be incomparable selective ultrafilters. Let $\{a_k^j: k \in \omega\} \in p_j$ be a strictly increasing sequence such that $a_k^j> k$ for each $k, j \in \omega$. Then there exists a family $\{I_j: j \in \omega\}$ of subsets of $\omega$ such that:
 \begin{enumerate}[(i)]
     \item $\{a_k^j: k \in I_j\} \in p_j$ for each $j \in \omega$;
     \item  $\{[k,a_k^j]: j \in \omega, k \in I_j\}$ are pairwise disjoint intervals of $\omega$.
 \end{enumerate}
 \end{Lema}

 \begin{Lema}[\cite{tomita1}]\label{abertosli}
 Let $G$ be a non-discrete Boolean topological group. Then there exist nonempty open sets $\{U_k^j: k \in \omega, j \in \omega\}$ such that if $u_k^j \in U_k^j$ for each $k, j \in \omega$, then $\{u_k^j: k, j \in \omega\}$ is linearly independent.
 \end{Lema}
 
 The following results ensure the existence of certain homomorphisms, necessary to construct the topological groups we want. Their proofs are based on \textbf{Lemma 3.7} and \textbf{Lemma 4.1} of \cite{tomita34}, and also \textbf{Lemma 4.1} of \cite{tomita2}.
 
 \begin{Lema}\label{ufselimcom}
 Let:
\begin{itemize}
    \item  $E$ be a countable subset of $2^\mathfrak{c}$ and $I \subset E$;
    \item  $F \subset E$ be a finite subset;
     \item for each $\xi \in I$, $k_{\xi} \in \omega$;
    \item $\{p_{\xi}: \xi \in I\}$ be a family of incomparable selective ultrafilters.
    \item for each $\xi \in I$,  $g_{\xi}:\omega \to ([E]^{<\omega})^{k_{\xi}}$ be a function so that $\{g_{\xi}^j(m): j < k_{\xi}, m \in \omega\}$ is linearly independent;
    \item for each $\xi \in I$, $d_{\xi} \in ([E]^{<\omega})^{k_{\xi}}$.
\end{itemize}
Then there exist an increasing sequence $\{b_i: i \in \omega\}\subset \omega$, a surjective function $r: \omega \to I$ and a sequence $\{E_i: i \in \omega\}$ of finite subsets of $E$ such that:
\begin{enumerate}[a)]
    \item $F \subset E_0$; 
    \item $E = \bigcup_{i \in \omega} E_i$;
    \item $r(m) \in E_m$ for each $m \in \omega$;
    \item $\bigcup \{d_{r(m)}^j: j< k_{r(m)}\} \subset E_m$, for each $m \in \omega$;

    \item $E_{m+1} \supset \bigcup(\{g_{\xi}^j(b_m): \xi \in E_m \cap I, j <k_{\xi}\} ) \cup E_m$, for each $m \in \omega$;
    \item $\{g_{r(m)}^j(b_m): j<k_{r(m)} \} \cup \{\{\mu\}: \mu \in E_m\} \text{ is linearly independent}$, for each $m \in \omega$;
    \item $\{b_i: i \in r^{-1}(\xi)\} \in p_{\xi}$, for every $\xi \in I$.
    ~\\
     Furthermore, if $\{y_n: n \in \omega\} \subset E$ is faithfully indexed, then $E_i$ can be arranged for each $i \in \omega$ so that
    \item $\{n \in \omega: y_n \in E_i\} = 2N_i$, for some $N_i \in \omega$, and $(N_i)_{i \in \omega}$ is a strictly increasing sequence.\footnotemark
\end{enumerate}
 \end{Lema}
 \footnotetext{For every $K \in \omega$, $K \geq 2$, we could also arrange $E_i$ for each $i \in \omega$ so that $\{n \in \omega: y_n \in E_i\} \subset KN_i$, for some $N_i \in \omega$, and  $(N_i)_{i \in \omega}$ is strictly increasing. The proof would be analogous.}
 \begin{proof}
Suppose first that $I$ is infinite. Let $E \doteq \{\xi_n: n \in \omega\}$ be an enumeration and $s: \omega \to \omega$ be a strictly increasing function such that $\{\xi_{s(j)}: j \in \omega\}= I$. We will first define a family $\{F_n: n \in \omega\}$ of finite subsets of $E$. This family will be used to construct the family $\{E_n: n \in \omega\}$. 

Choose $N_0 \in \omega$ so that $\left\{n \in \omega: y_n \in F \cup \{\xi_0\} \cup (\bigcup \{d_{\xi_{s(0)}}^j: j<k_{\xi_{s(0)}} \}) \right\} \subset 2N_0$, and define
\[
F_0 \doteq \{y_n: n \leq 2N_0\} \cup F \cup \{\xi_0\} \cup (\bigcup \{d_{\xi_{s(0)}}^j: j<k_{\xi_{s(0)}} \}).
\]

Suppose that we have defined finite subsets $F_0,...,F_{l} \subset E$ so that
\begin{enumerate}[1)]
\item $\xi_p \in F_p$ for each $0 \leq p\leq l$;
    \item $F_{p+1} \supset \bigcup(\{g_{\beta}^j(m): m \leq p, \ \beta \in F_{p}\cap I, \ j< k_{\beta}\}) \cup F_{p}$ for each $0 \leq p<l$.
    
    \item $\bigcup \{d_{\xi_{s(p)}}^j: j < k_{\xi_{s(p)}}\} \subset F_p$, for each $0 \leq p \leq l$.
    
    \item $\{n \in \omega: y_n \in F_p\} = 2N_p$, for some $N_p \in \omega$, for each $0 \leq p \leq l$.
\end{enumerate}
Now choose $N_{l+1}>N_{l}$ so that 
\[
\left\{n \in \omega: y_n \in  \bigcup \Big(\{g_{\beta}^j(m): m \leq l, \ \beta \in F_{l}\cap I, \ j< k_{\beta}\} \cup \{d_{\xi_{s(l+1)}}^j: j < k_{\xi_{s(l+1)}}\}\Big) \cup F_{l} \cup \{\xi_{l+1}\}\right\} \subset 2N_{l+1},
\]
and then define 
\[
F_{l+1} \doteq \{y_n: n \leq 2N_{l+1}\} \cup \bigcup \Big(\{g_{\beta}^j(m): m \leq l, \ \beta \in F_{l}\cap I, \ j< k_{\beta}\} \cup \{d_{\xi_{s(l+1})}^j: j < k_{\xi_{s(l+1)}}\} \Big)  \cup F_l \cup \{\xi_{l+1}\}.
\]

It is clear that 1), 2), 3) and 4) are also satisfied for $F_0, ..., F_{l+1}$. Then, we may construct recursively a family $\{F_n: n \in \omega\}$ of finite subsets of $E$ satisfying 1)-4) for every $p \in \omega$. We also have that $E = \bigcup_{i \in \omega}F_i$.

For each $\xi \in I$ and $n \in \omega$, let

\[
    A_n^{\xi} \doteq 
    \{m \in \omega: \{g_{\xi}^j(m): j<k_{\xi} \} \cup \{\{\mu\}: \mu \in F_n\} \text{ is linearly independent}\}.
\]

Since $\{g_{\xi}^j(m): j<k_{\xi} , m \in \omega\}$ is linearly independent and $F_n$ is finite, we have that $A_n^{\xi}$ is cofinite, and then $A_n^{\xi} \in p_{\xi}$, for every $n \in \omega$ and $\xi \in I$. Since selective ultrafilters are $P-$points, for each $\xi \in I$ there exists $A_{\xi} \in p_{\xi}$ so that $A_\xi \setminus A_n^{\xi} \text{ is finite for every } n \in \omega$. 

Now, for each $\xi \in I$, let $v_{\xi}: \omega \to \omega$ be a strictly increasing function so that $A_{\xi} \setminus A_n^{\xi} \subset v_{\xi}(n), \text{ for each } n \in \omega.$ As every $p_{\xi}$ is a selective ultrafilter, for each $\xi \in I$ there exists $B_{\xi} \in p_{\xi}$ such that
\[
B_{\xi} \cap v_{\xi}(1) = \emptyset, B_{\xi} \subset A_{\xi} \text{ and } |[v_{\xi}(n)+1,v_{\xi}(n+1)] \cap B_{\xi}| \leq 1, \text{ for each } n \in \omega.
\]

Let $\{a_n^{\xi}: n \in \omega\}$ be the strictly increasing enumeration of $B_{\xi}$, for each $\xi \in I$. Notice that $a_n^{\xi} > v_{\xi}(n) \geq n$ for each $n \in \omega$ and $\xi \in I$. Thus,
\[
a_n^{\xi} \in A_n^{\xi}, \text{ for each } \xi \in I  \text{ and } n \in \omega,
\]
and, by Lemma \ref{enumeravelincom}, there exists a family $\{I_{\xi}: \xi \in I\}$ of subsets of $\omega$ such that:
\begin{enumerate}[i)]
\item $\{a_i^{\xi}: i \in I_{\xi}\} \in p_{\xi}$ for each $\xi \in I$;
\item $\{[i,a_{i}^{\xi}]: \xi \in I \text{ and } i \in I_{\xi}\}$ are pairwise disjoint intervals of $\omega$.
\end{enumerate}

By ii), the sets $\{I_{\xi}: \xi \in I\}$ are pairwise disjoint. We may also assume without loss of generality that $I_{\xi_{s(k)}} \subset \omega \setminus s(k)$ for every $k \in \omega$. Let $\{i_m: m \in \omega\}$ be the strictly increasing enumeration of $\bigcup_{n \in \omega}I_{\xi_{s(n)}}$ and $r: \omega \to I$ be such that $r(m)=\xi_{s(i)}$ if and only if $i_m \in I_{\xi_{s(i)}}$. Define also $b_{m} \doteq a_{i_m}^{r(m)}$ and $E_m \doteq F_{i_m}$, for each $m \in \omega$.

Conditions a) and b) are trivially satisfied. Moreover, given $m \in \omega$, if $i_m \in I_{\xi_{s(i)}}$, then $i_m \geq s(i)$, and hence $r(m) \in E_m$. Therefore, conditions c) and d) are satisfied. To check condition e), note that $b_m = a_{i_m}^{r(m)} \leq i_{m+1}-1$ and $E_m = F_{i_m} \subset F_{i_{m+1}-1}$ for each $m \in \omega$, thus
\begin{align*}
E_m &\cup \bigcup (\{g_{\xi}^j(b_m): \xi \in E_m \cap I, j< k_{\xi}\} ) \\ &\subset
F_{i_{m+1}-1} \cup \bigcup (\{g_{\xi}^j(p): p \leq i_{m+1}-1, \ \xi \in F_{i_{m+1}-1} \cap I, \ j< k_{\xi}\}  ) \\
&\subset F_{i_{m+1}}= E_{m+1}.
\end{align*}

Condition f) is also satisfied, since $b_m = a_{i_m}^{r(m)} \in A_{i_m}^{r(m)}$ for each $m \in \omega$, and hence,
\[
    \{g_{r(m)}^j(b_m): j<k_{r(m)} \} \cup \{\{\mu\}: \mu \in F_{i_m}\} \text{ is linearly independent.}
\]
To check condition g), simply note that, given $\xi \in I$, 
\[
\{b_m: m \in r^{-1}(\xi)\} = \{a_{i}^{\xi}: i
\in I_{\xi} \} \in p_{\xi}.
\]
Condition $h)$ follows by construction.

If $I$ is finite, the proof is basically the same, replacing the use of Lemma \ref{enumeravelincom} by Lemma \ref{finitolincom}.
\end{proof}

 \begin{Lema}\label{lemaunitario}
 Let:
\begin{itemize}
    \item  $Z_0$ and $Z_1$ be disjoint countable subsets of $2^\mathfrak{c}$, and $E= Z_0 \cupdot Z_1$;
    \item  $I_0 \subset Z_0$, $I_1 \subset Z_1$, and $I \doteq I_0 \cupdot I_1$;
       \item $\mathcal{F} \subset [E]^{<\omega}$ be a finite linearly independent subset and, for each $f \in \mathcal{F}$, let $n_f \in 2$;
     \item for each $\xi \in I$, $k_{\xi} \in \omega$;
    \item $\{p_{\xi}: \xi \in I\}$ be a family of incomparable selective ultrafilters;
    \item for each $\xi \in I$, $\delta_{\xi}=0$ if $\xi \in I_0$ and $\delta_{\xi}=1$ if $\xi \in I_1$;
    \item for every $\xi \in I$,  $g_{\xi}:\omega \to ([Z_{\delta_{\xi}}]^{<\omega})^{k_{\xi}}$ be a function so that $\{g_{\xi}^j(m): j < k_{\xi}, m \in \omega\}$ is linearly independent;
    \item for each $\xi \in I$, $d_{\xi} \in ([Z_{\delta_{\xi}}]^{<\omega})^{k_{\xi}}$;
    \item $\{z^0_n: n \in \omega\} \subset Z_0$ and $\{z_n^1: n \in \omega\} \subset Z_1$ be sequences of pairwise distinct elements.
\end{itemize}
Then, given $(\alpha_0, \alpha_1) \in 2 \times 2$, there exists a homomorphism $\Phi:[E]^{<\omega} \to 2$ such that:
\begin{enumerate}[(i)]
\item $\Phi(f)=n_f$, for every $f \in \mathcal{F}$;
\item for every $\xi \in I$, 
\[\left\{n \in \omega: \left(\Phi(g_{\xi}^0(n)),...,\Phi(g_{\xi}^{k_{\xi}-1}(n))\right) = \left(\Phi(d_{\xi}^0),...\Phi(d_{\xi}^{k_{\xi}-1})\right)\right\} \in p_{\xi};\]

\item $\{n \in \omega: \big(\Phi(\{z^0_n\}), \Phi(\{z^1_n\})\big)=(\alpha_0,\alpha_1)\} \text{ is finite.}$
 
\end{enumerate} 
 \end{Lema}
 \begin{proof}
 Firstly we apply Lemma \ref{ufselimcom} using the elements given in the hypothesis, $F = \bigcup \mathcal{F}$, and the following sequence $y: \omega \to E$ for item h): for each $n \in \omega$, write $n = 2q+j$ for the unique $q \in \omega$ and $j \in 2$, and put
\[
y_{2q+j} =
 \begin{cases}
    z_q^0, & \text{if} \ j=0\\
    z_q^1, & \text{if } j=1.
\end{cases}
\]
Thus we obtain $\{b_i: i \in \omega\} \subset \omega$, $r:\omega \to I$ and $\{E_m: m \in \omega\} \subset [E]^{<\omega}$ satisfying a)--h).
 
 We shall define auxiliary homomorphisms $\Phi_m:[E_m]^{<\omega} \to 2$ inductively. First, we define $\Phi_0:[E_0]^{<\omega} \to 2$ so that $\Phi_0(f)=n_f$ for each $f \in \mathcal{F}$. Now, suppose that, for $l \in \omega$, we have defined homomorphisms $\Phi_m:[E_m]^{<\omega} \to 2$ for each $m=0,...,l$, so that
\begin{enumerate}[(1)]
    \item $\Phi_{m+1}$ extends $\Phi_{m}$ for each $0 \leq m<l$;
    \item for every $0 \leq m<l$,
 \[
    \big(\Phi_{m+1}(g^0_{r(m)}(b_{m})),...,\Phi_{m+1}(g^{k_{r(m)}-1}_{r(m)}(b_m))\big) = \big(\Phi_{m}(d_{r(m)}^0),...,\Phi_{m}(d_{r(m)}^{k_{r(m)}-1})\big);
    \]
    \item $\big(\Phi_m(\{z_n^0\}),\Phi_m(\{z_n^1\})\big) \neq (\alpha_0,\alpha_1)$ for each $0 < m \leq l$ and $n \in \omega$ so that $z_n^0 \in E_m \setminus E_{m-1}$.
\end{enumerate}
 
 We shall prove that we may define $\Phi_{l+1}:[E_{l+1}]^{<\omega} \to 2$ so that $\Phi_0,...,\Phi_{l+1}$ also satisfy (1), (2) and (3). For this, suppose without loss of generality that $r(l) \in I_0$. By item f) of Lemma \ref{ufselimcom}, $\{g_{r(l)}^j(b_l): j < k_{r(l)}\} \cup \{\{\mu\}: \mu \in E_l\}$ is linearly independent, and, by item h), for every $n, m \in \omega$, $z_n^0 \in E_{m}$ if, and only if, $z_n^1 \in E_m$. Since $g_{r(l)}^j(b_l) \in [Z_0]^{<\omega}$ for every $j<k_{r(l)}$, and $z_n^1 \in Z_1$ for every $n \in \omega$, we conclude that
\begin{equation*} \tag{\dag}
\{\{z_n^1\}:  z_n^0 \in E_{l+1} \setminus E_l\} \cup \{g_{r(l)}^j(b_{l}): j<k_{r(l)}\} \cup \{\{\mu\}: \mu \in E_{l}\}
\end{equation*} 
is linearly independent. Therefore, using items d) and e) of Lemma \ref{ufselimcom}, we may define $\Phi_{l+1}:[E_{l+1}]^{<\omega} \to 2$ extending $\Phi_l$ so that 
\[
(\Phi_{l+1}(g_{r(l)}^0(b_{l})),...,\Phi_{l+1}(g_{r(l)}^{k_{r(l)}-1}(b_{l}))) = (\Phi_{l}(d_{r(l)}^{0}),...,\Phi_{l}(d_{r(l)}^{k_{r(l)}-1}))
\]
and 
\begin{equation*} \tag{\ddag}
\Phi_{l+1}(\{z_n^1\}) \neq \alpha_1
\end{equation*}
for each $n \in \omega$ such that $z_n^0 \in E_{l+1} \setminus E_l$. Thus, we have that $\Phi_0,...,\Phi_{l+1}$ also satisfy (1), (2) and (3), and therefore there exists a sequence $(\Phi_m)_{m \in \omega}$ of homomorphisms $\Phi_m:[E_m]^{<\omega} \to 2$ satisfying these properties.

We claim that the homomorphism $\Phi \doteq \bigcup_{n \in \omega} \Phi_n : [E]^{<\omega} \to 2$ satisfies the hypothesis we want. In fact, items (i) and (iii) are clear from the construction and item (ii) follows from the fact that for every $\xi \in I$,
\[
(\Phi(g_{\xi}^0(b_i)),...,\Phi(g_{\xi}^{k_{\xi}-1}(b_i))) = (\Phi(d_{\xi}^0),..., \Phi(d_{\xi}^{k_{\xi}-1})),
\]
 for each $i \in r^{-1}(\xi)$, and that $\{b_i: i \in r^{-1}(\xi) \} \in p_{\xi}$, by item g) of Lemma \ref{ufselimcom}. 
 \end{proof}
 
 The next result is a stronger version of the previous lemma, and uses it in its proof.
 
\begin{Lema}\label{completo}
 Let:
\begin{itemize}
    \item  $Z_0$ and $Z_1$ be disjoint countable subsets of $2^\mathfrak{c}$, and $E= Z_0 \cupdot Z_1$;
    \item  $I_0 \subset Z_0$, $I_1 \subset Z_1$, and $I \doteq I_0 \cupdot I_1$;
     \item $\mathcal{F} \subset [E]^{<\omega}$ be a linearly independent finite subset and, for each $f \in \mathcal{F}$, let $n_f \in 2$;
     \item for each $\xi \in I$, $k_{\xi} \in \omega$;
    \item $\{p_{\xi}: \xi \in I\}$ be a family of incomparable selective ultrafilters;
    \item for each $\xi \in I$, $\delta_{\xi}=0$ if $\xi \in I_0$ and $\delta_{\xi}=1$ if $\xi \in I_1$;
    \item for every $\xi \in I$,  $g_{\xi}:\omega \to ([Z_{\delta_{\xi}}]^{<\omega})^{k_{\xi}}$ be a function so that $\{g_{\xi}^j(m): j < k_{\xi}, \ m \in \omega\}$ is linearly independent;
    \item for each $\xi \in I$, $d_{\xi} \in ([Z_{\delta_{\xi}}]^{<\omega})^{k_{\xi}}$;
    \item $\{y^0_n: n \in \omega\} \subset [Z_0]^{<\omega}$ and $\{y_n^1: n \in \omega\} \subset [Z_1]^{<\omega}$ be linearly independent subsets.
\end{itemize}

Suppose that $|Z_i \setminus \bigcup \{y_n^i: n \in \omega\}| = \omega$, for each $i \in 2$. Then, given $(\alpha_0,\alpha_1) \in 2 \times 2$, there exists a homomorphism $\Phi:[E]^{<\omega} \to 2$ such that:
\begin{enumerate}[(i)]
\item $\Phi(f)=n_f$, for every $f \in \mathcal{F}$;
\item for every $\xi \in I$, 
\[\left\{n \in \omega: \left(\Phi(g_{\xi}^0(n)),...,\Phi(g_{\xi}^{k_{\xi}-1}(n))\right) = \left(\Phi(d_{\xi}^0),...,\Phi(d_{\xi}^{k_{\xi}-1})\right)\right\} \in p_{\xi};\]

\item $\{n \in \omega: (\Phi(y^0_n), \Phi(y_n^1))=(\alpha_0,\alpha_1)\} \text{ is finite.}$
 \end{enumerate}  
 \end{Lema}
 \begin{proof}
For each $i \in 2$, let $\{z_n^i: n \in \omega\}$ be an enumeration of $\bigcup \{y_n^i: n \in \omega\}$. Next, we extend $\{y_n^i: n \in \omega\}$ to a basis $\mathcal{B}^i$ of $[Z_i]^{<\omega}$ and also  $\{\{z_n^i\}: n \in \omega\}$ to a basis $\mathcal{C}^i$ of $[Z_i]^{<\omega}$, for each $i \in 2$. By assumption, $|\mathcal{C}^i \setminus \{\{z_n^i\}: n \in \omega\} | = |\mathcal{B}^i \setminus \{y_n^i: n \in \omega\}| = \omega$, thus we may consider enumerations $\{e_k^i: k \in \omega \}$ of $\mathcal{C}^i \setminus \{\{z_n^i\}: n \in \omega\}\}$ and $\{f_k^i: k \in \omega \}$ of $\mathcal{B}^i \setminus \{y_n^i: n \in \omega\}$. It is clear that both $\mathcal{B}^0 \cup \mathcal{B}^1$ and $\mathcal{C}^0 \cup \mathcal{C}^1$ are basis of $[E]^{<\omega}$. 
 
Let $\theta: [E]^{<\omega} \to [E]^{<\omega}$ be the isomorphism defined by 
\[
\theta(y_n^i) = \{z_n^i\},
\]
and
\[
\theta(f_k^i) = e_k^i,
\]
for each $i \in2$ and $n, k \in \omega$. Note that $\restr{\theta}{[Z_i]^{<\omega}}:[Z_i]^{<\omega}\to [Z_i]^{<\omega}$ is also an isomorphism, for each $i \in 2$.

Let, for every $\xi \in I$, $h_{\xi} : \omega \to ([Z_{\delta_{\xi}}]^{<\omega})^{k_{\xi}}$ be given by $h_{\xi}^j(n)= \theta(g_{\xi}^j(n))$ for each $n \in \omega$ and $j<k_{\xi}$, and $\overline{d}_{\xi} \in ([Z_{\delta_{\xi}}]^{<\omega})^{k_{\xi}}$ be given by $\overline{d}_{\xi}^j = \theta(d_{\xi}^j)$, for each $j<k_{\xi}$. 

By Lemma \ref{lemaunitario}, there exists a homomorphism $\overline{\Phi}:[E]^{<\omega} \to 2$ so that:
\begin{enumerate}[(i)]
\item $\overline{\Phi}(\theta(f))=n_f$, for every $f \in \mathcal{F}$;
\item For every $\xi \in I$, $\left\{n \in \omega: \left(\overline{\Phi}(h_{\xi}^0(n)),...,\overline{\Phi}(h_{\xi}^{k_{\xi}-1}(n))\right) = \left(\overline{\Phi}(\overline{d}_{\xi}^0),...\overline{\Phi}(\overline{d}_{\xi}^{k_{\xi}-1})\right)\right\} \in p_{\xi}$;

\item $\{n \in \omega: (\overline{\Phi}(\{z^0_n\}), \overline{\Phi}(\{z^1_n\}))=(\alpha_0,\alpha_1)\} \text{ is finite.}$
 
\end{enumerate} 

Thus, the homomorphism $\Phi \doteq \overline{\Phi} \circ \theta:[E]^{<\omega} \to 2$ satisfies the hypothesis we want.

 \end{proof}
  \begin{Rem}\label{RemLema}
 Note that in the statement of the previous lemma, item (iii) can be replaced by the following (stronger) condition, for a given $\alpha \in 2$:
 \[
 \textit{(iii) } \{n \in \omega: \Phi(y_n^0 \triangle y_n^1) = \alpha\} \text{ is finite.}
 \]
 
 Indeed, we could replace condition ($\dag$) in the proof of Lemma \ref{lemaunitario} by the fact that
 \[\{\{z_n^0, z_n^1\}:  z_n^0 \in E_{l+1} \setminus E_l\} \cup \{g_{r(l)}^j(b_{l}): j<k_{r(l)}\} \cup \{\{\mu\}: \mu \in E_{l}\}
  \]
  is linearly independent, thus in equation ($\ddag$) we could choose
 \[\Phi_{l+1}(\{z_n^0, z_n^1\}) \neq \alpha\]
 for each $n \in \omega$ such that $z_n^0 \in E_{l+1} \setminus E_l$. Then, the proof of Lemma \ref{completo} would remain the same, just replacing the old condition with the new one when required.
 \end{Rem}

The next result is an easy corollary of the previous lemma.
 
 \begin{Corol}\label{homfraco}
 Let:
 \begin{itemize}
     \item $E$ be a countable subset of $2^\mathfrak{c}$;
     \item $I \subset E$;
     \item $\mathcal{F} \subset [E]^{<\omega}$ be a linearly independent finite subset and, for each $f \in \mathcal{F}$, let $n_f \in 2$;
     \item for each $\xi \in I$, $k_{\xi} \in \omega$.
     \item $\{p_{\xi}: \xi \in I\}$ be a family of incomparable selective ultrafilters;
     \item for every $\xi \in I$, $g_{\xi}: \omega \to ([E]^{<\omega})^{k_{\xi}}$ be a function so that $\{g_{\xi}^j(m): j< k_{\xi}, \ m \in \omega\}$ is linearly independent;
     \item for every $\xi \in I$, $d_{\xi} \in ([E]^{<\omega})^{k_{\xi}}$.
     \end{itemize}
      Then there exists a homomorphism $\Phi:[E]^{<\omega} \to 2$ such that:
\begin{enumerate}[(i)]
\item $\Phi(f)=n_f$, for every $f \in \mathcal{F}$;
\item For every $\xi \in I$, $\left\{n \in \omega: \left(\Phi(g_{\xi}^0(n)),...,\Phi(g_{\xi}^{k_{\xi}-1}(n))\right) = \left(\Phi(d_{\xi}^0),...\Phi(d_{\xi}^{k_{\xi}-1})\right)\right\} \in p_{\xi}$.
 \end{enumerate}  
     \end{Corol}

     Although the proof of the following result is similar to the proof of Lemma \ref{lemaunitario} and Lemma \ref{completo}, we present it here for the sake of completeness.
      
    \begin{Lema}\label{caso2.2}
     Let:
 \begin{itemize}
     \item $E$ be a countable subset of $2^\mathfrak{c}$;
     \item $I \subset E$;
      \item $\mathcal{F} \subset [E]^{<\omega}$ be a linearly independent finite subset and, for each $f \in \mathcal{F}$, let $n_f \in 2$;
     \item $n \in \omega$;
     \item $\{p_{\xi}: \xi \in I\}$ be a family of incomparable selective ultrafilters;
     \item for every $\xi \in I$, $g_{\xi}: \omega \to ([E]^{<\omega})^{n}$ be a function so that $\{g_{\xi}^j(m): j< n, \ m \in \omega\}$ is linearly independent;
     \item for every $\xi \in I$, $d_{\xi} \in ([E]^{<\omega})^{n}$;
      \item $\{y^j_k: k \in \omega, \ j\leq n\} \subset [E]^{<\omega}$ be a linearly independent subset.
     \end{itemize}
       Suppose that $|E \setminus \bigcup \{y_k^j: k \in \omega, \ j \leq n\}|= \omega$. Then, given $(\alpha_0,...,\alpha_n) \in 2^{n+1}$, there exists a homomorphism $\Phi:[E]^{<\omega} \to 2$ such that:
\begin{enumerate}[(i)]
\item $\Phi(f)=n_f$, for every $f \in \mathcal{F}$;
\item for every $\xi \in I$, $\left\{k \in \omega: \left(\Phi(g_{\xi}^0(k)),...,\Phi(g_{\xi}^{n-1}(k))\right) = \left(\Phi(d_{\xi}^0),...\Phi(d_{\xi}^{n-1})\right)\right\} \in p_{\xi}$;
\item $\{k \in \omega: (\Phi(y_k^0),...,\Phi(y_k^{n}))= (\alpha_0,....,\alpha_n)\}$ is finite.
 \end{enumerate}  
     \end{Lema}
     \begin{proof} We split the proof in two cases.
     
     \textbf{Case 1:} Suppose that each $y_k^j$ is a singleton, that is, $y_k^j=\{z_k^j\}$, for some $z_k^j \in E$, for every $j \leq n$ and $k \in \omega$.
     ~\\
     
     In this case, we apply Lemma \ref{ufselimcom} using the elements of the statement, $F \doteq \bigcup \mathcal{F}$, $k_{\xi}=n$ for each $\xi \in I$, and the following sequence $w: \omega \to E$ in item h): for each $m \in \omega$, write $m=(n+1)q+j$ for the unique $q \in \omega$ and $j \in (n+1)$, and put $w_{m}= z_q^j$. Thus, we obtain $\{b_i: i \in \omega\}$, $r: \omega \to I$ and $\{E_m:m \in \omega\} \subset [E]^{<\omega}$ satisfying a)-h) of this lemma. 
     
     We shall again define auxiliary homomorphisms $\Phi_m: [E_m]^{<\omega} \to 2$, for each $m \in \omega$, inductively. First, define $\Phi_0:[E_0]^{<\omega} \to 2$ so that $\Phi_0(f)=n_f$, for each $f \in \mathcal{F}$. Suppose that, for $l \in \omega$, we have defined $\Phi_m:[E_m]^{<\omega} \to 2$, for each $m=0,...,l$, satisfying that:
     \begin{enumerate}[(1)]
         \item $\Phi_{m+1}$ extends $\Phi_m$, for each $0 \leq m <l$;
         \item for every $0 \leq m <l$,
         \[
         \big(\Phi_{m+1}(g_{r(m)}^0(b_m)),...,\Phi_{m+1}(g_{r(m)}^{n-1}(b_m))\big) = \big(\Phi_m(d_{r(m)}^0),..., \Phi_m(d_{r(m)}^{n-1}) \big);
         \]
         \item $\big(\Phi_m(\{z_k^0\}),...\Phi_m(\{z_k^n\})\big) \neq \big(\alpha_0,...,\alpha_n\big)$ for each $0<m \leq l$ and $k \in \omega$ so that $z_k^0 \in E_m \setminus E_{m-1}$\footnotemark.\footnotetext{Recall that, by construction, given $k, m \in \omega$, $z_k^j \in E_m$ for \textbf{some} $0 \leq j \leq n$ if, and only if, $z_k^j \in E_m$ for \textbf{every} $0 \leq j \leq n$.} 
     \end{enumerate}
     
     Now, since by construction $\{g_{r(l)}^j(b_l): j<n\} \cup \{\{\mu\}: \mu \in E_l\} $ is linearly independent, we may apply Lemma \ref{novo} with $A \doteq \{g_{r(l)}^j(b_l): j<n\}$, $B \doteq \{\{z_k^j\}: z_k^0 \in E_{l+1}, \setminus E_l, \ j \leq n\}$ and $C \doteq \{\{\mu\}: \mu \in E_l\}$ to obtain a subset $B' \subset B$ such that $|B'| \leq |A| = n$ and 
     \[
     \{g_{r(l)}^j(b_l): j<n\} \cup \{\{\mu\}: \mu \in E_l\} \cup (\{\{z_k^j\}: z_k^0 \in E_{l+1} \setminus E_l, \ j \leq n\} \setminus B')
     \]
     is linearly independent. Then, for each $k \in \omega$ so that $z_k^0 \in E_{l+1} \setminus E_l$, there exists $0 \leq j^k \leq n$ such that $z_k^{j^k} \in (\{\{z_k^{j}\}: z_k^0 \in E_{l+1} \setminus E_l, \ j \leq n\} \setminus B')$. Thus, we may define $\Phi_{l+1}:[E_{l+1}]^{<\omega} \to 2$ extending $\Phi_l$ so that
     \[
     \big(\Phi_{l+1}(g_{r(l)}^0(b_l)),..., \Phi_{l+1}(g_{r(l)}^{n-1}(b_l))\big) = \big(\Phi_{l}(d^0_{r(l)}),...,\Phi_{l}(d_{r(l)}^{n-1})\big)
     \]
     and
     \[
     \Phi_{l+1}(\{z_k^{j^k}\}) \neq \alpha_{j^k},
     \]
     for every $k \in \omega$ so that $z_k^{0} \in E_{l+1} \setminus E_l $. Similarly to the proof of Lemma \ref{lemaunitario}, we have that $\Phi_0,...,\Phi_{l+1}$ also satisfy (1)-(3), and therefore there exists a sequence $(\Phi_m)_{m \in \omega}$ of homomorphisms $\Phi_m:[E_m]^{<\omega} \to 2$ satisfying such properties. Again, the homomorphism defined by $\Phi \doteq \bigcup_{n \in \omega} \Phi_n : [E]^{<\omega} \to 2$ satisfies the hypothesis we want.
    \vspace{0.4 cm} 

\textbf{Case 2:} The general case. There is no restriction on elements $y_k^j$. 

Let $\{z_k: k \in \omega\}$ be an enumeration of $\bigcup \{y_k^j: k \in \omega, j \leq n\}$ and $\{z_k^0: k \in \omega\}, ... , \{z_k^n: k \in \omega\}$ be a partition of $\{z_k: k \in \omega\}$. We extend $\{y_k^j: k \in \omega, j \leq n\}$ to a basis $\mathcal{B}$ of $[E]^{<\omega}$ and also  $\{\{z_k^j\}: k \in \omega, j \leq n\}$ to a basis $\mathcal{C}$ of $[E]^{<\omega}$. By assumption, $|\mathcal{B} \setminus \{\{z_k^j\}: k \in \omega, j \leq n\}| = |\mathcal{C} \setminus \{y_k^j: k \in \omega, j \leq n\}|= \omega$, thus consider enumerations $\{e_l: l \in \omega \}$ of $\mathcal{B} \setminus \{\{z_k^j\}: k \in \omega, j \leq n\}$ and $\{f_l: l \in \omega \}$ of $\mathcal{C} \setminus \{y_k^j: k \in \omega, j \leq n\}$.
     
     Let $\theta:[E]^{<\omega} \to [E]^{<\omega}$ be the isomorphism defined by:
     \[
     \theta(y_k^j) = \{z_k^j\},
     \]
     for every $k \in \omega$ and $j \leq n$, and
     \[
     \theta(f_l)=e_l,
     \]
     for every $l \in \omega$.
     
     Let also, for each $\xi \in I$, $h_{\xi}:\omega \to ([E]^{<\omega})^n$ given by $h_{\xi}^i(m) = \theta(g_{\xi}^i(m))$, for every $m \in \omega$ and $i < n$, and $\overline{d_{\xi}} \in ([E]^{<\omega})^n$ given by $\overline{d_{\xi}}^i =\theta(d_{\xi}^i)$, for every $i <n$. By the previous case, there exists a homomorphism $\tilde{\Phi}:[E]^{<\omega} \to 2$ so that:
     \begin{enumerate}[(i)]
         \item $\tilde{\Phi}(\theta(f)) = n_f$, for each $f \in \mathcal{F}$;
         \item For every $\xi \in I$, $\left\{m \in \omega: \left(\overline{\Phi}(h_{\xi}^0(m)),...,\overline{\Phi}(h_{\xi}^{n-1}(m))\right) = \left(\overline{\Phi}(\overline{d}_{\xi}^0),...\overline{\Phi}(\overline{d}_{\xi}^{n-1})\right)\right\} \in p_{\xi}$;
         
         \item $\{k \in \omega: (\overline{\Phi}(\{z^0_k\}),..., \overline{\Phi}(\{z^n_k\}))=(\alpha_0,...,\alpha_n)\} \text{ is finite.}$
     \end{enumerate}
     
     Thus, the homomorphism $\Phi \doteq \overline{\Phi} \circ \theta:[E]^{<\omega} \to 2$ satisfies the hypothesis we want.

      \end{proof}

\section{A consistent solution to the case $\alpha= \omega$ of the Comfort-like question for countably pracompact groups}

\begin{Teo}\label{casoomega}
Suppose that there are $\mathfrak{c}$ incomparable selective ultrafilters. Then there exists a (Hausdorff) group $G$ which has all finite powers countably pracompact and such that $G^{\omega}$ is not countably pracompact. 
\end{Teo}
\begin{proof}
The required group will be constructed giving a suitable topology to the Boolean group $[\mathfrak{c}]^{<\omega}$, as follows.

Let $(X_n)_{n > 0}$ be a partition of $\mathfrak{c}$ so that $|X_n| = \mathfrak{c}$ for every $n > 0$. For each $n > 0$, let $(X_n^j)_{j <2}$ be a partition of $X_n$ so that
\begin{itemize}
    \item $|X_n^0|=|X_n^1| = \mathfrak{c}$;
    \item $X_n^0$ contains only limit ordinals and their next $\omega$ elements;
    \item the initial $\omega$ elements of $X_n$ are in $X_n^1$.
\end{itemize}

For every $n > 0$, let also
\[Y_n^0 \doteq \{\xi \in X_n^0: \xi \text{ is a limit ordinal}\},\]
and define the sets $X_0 \doteq \bigcupdot_{n \in \omega} X_n^0$, $X_1 \doteq \bigcupdot_{n \in \omega} X_n^1$ and  $Y_0 \doteq \bigcupdot_{n \in \omega} Y_n^0$. Now, consider a family of functions $\{f_{\xi}: \xi \in Y_0\}$ so that: 
\begin{enumerate}[1)]
  \item for each $n >0,$ $\{f_{\xi}: \xi \in Y_n^0\}$ is an  enumeration of all the sequences $(x_k)_{k \in \omega}$ of elements in $([X_n]^{<\omega})^n$ so that $\{x_k^j: k \in \omega, j<n\}$ is linearly independent;
   \item  given $n > 0$ and $\xi \in Y_n^0$, $f_{\xi}$ is a function from $\omega$ to $([X_n]^{<\omega})^n$ such that $\bigcup_{j<n} \bigcup_{k \in \omega} f_{\xi}^j(k) \subset \xi$.
\end{enumerate}
Finally, let $\{p_{\xi}: \xi \in Y_0\}$ be a family of incomparable selective ultrafilters, which exists by hypothesis.

Countable subsets of $\mathfrak{c}$ which have a suitable property of closure related to this construction will be called \textit{suitably closed}\footnotemark:
\footnotetext{The idea of suitably closed sets already appeared in \cite{koszmider&tomita&watson}, without using a name. Many subsequent works that used Martin's Axiom for countable posets and selective ultrafilters also used this idea. The name \textit{suitably closed} appeared firstly in \cite{Michael}.}

\begin{Def}
A set $A \in [\mathfrak{c}]^{\omega}$ is \textit{suitably closed} if, for each $n > 0$ and $\xi \in Y_n^0$ so that $\{\xi+j: j <n \} \cap  A \neq \emptyset$, we have that 
\[\{\xi+j: j < n\} \cup \bigcup_{j<n} \bigcup_{k \in \omega} f_{\xi}^j(k) \subset A.\]
\end{Def}

Let $\mathcal{A}$ be the set of all homomorphisms $\sigma: [A]^{< \omega} \to 2$, with $A \in [\mathfrak{c}]^{\omega}$ suitably closed, satisfying that, for every  $n >0$ and $\xi \in A \cap Y_n^0$,
\[
\sigma(\{\xi + j\}) = p_{\xi}-\lim_{k \in \omega} \sigma(f_{\xi}^j(k)),
\]
for each $j<n$.

Enumerate $\mathcal{A}$ by $\{\sigma_{\mu}: \omega \leq \mu < \mathfrak{c}\}$ and, without loss of generality, we may assume that  $\bigcup \text{dom}(\sigma_{\mu}) \subset \mu$, for each $\mu \in [\omega, \mathfrak{c})$. In what follows, we will construct suitable homomorphisms $\overline{\sigma_{\mu}}:[\mathfrak{c}]^{< \omega} \to 2$, for every $\mu \in [\omega, \mathfrak{c})$. Note that it is enough to define $\overline{\sigma_{\mu}}$ in the subset $\{\{\xi\}: \xi \in \mathfrak{c}\}$, since this is a basis for $[\mathfrak{c}]^{<\omega}$.

Firstly, for each $n > 0$, we enumerate all functions $g: S \to 2$ with $S \in [\mathfrak{c}]^{<\omega}$ by $\{g_{\xi}: \xi \in X_n^1 \}$. Without loss of generality, we may assume that $\text{dom}(g_{\xi}) \subset \xi$, for every $\xi \in X_n^1$, and that for each $g :S \to 2$ as above, $|\{\xi \in X_n^1:g_{\xi} = g\}| = \mathfrak{c}$. 

Let $\mu \in [\omega, \mathfrak{c})$. If $\xi < \mathfrak{c}$ is such that $\{\xi\} \in \text{dom}(\sigma_\mu)$, we put $\overline{\sigma}_{\mu}(\{\xi\})= \sigma_{\mu}(\{\xi\})$. Otherwise, we have a few cases to consider:
\begin{enumerate}[1)]
    \item  if $\xi \in X_1$ and $\mu \in \text{dom}(g_{\xi})$, we put $\overline{\sigma}_{\mu}(\{\xi\}) = g_{\xi}(\mu)$;
    \item if $\xi \in X_1$ and $\mu \notin \text{dom}(g_{\xi})$, we put $\overline{\sigma}_{\mu}(\{\xi\})=0$;
    \item for the remaining elements of $X_0$, $\overline{\sigma}_{\mu}$ is defined recursively, by putting
\[
\begin{cases}
\overline{\sigma_{\mu}}(\{\xi+j\}) = p_{\xi}-\lim_{k \in \omega} \overline{\sigma_{\mu}}(f_{\xi}^j(k)) & \text{if }  \xi \in Y_n^0 \text{ and } j <n;\\
        \overline{\sigma}_{\mu}(\{\xi\}) = 0,    & \text{if } \xi \notin \{\alpha + j: \alpha \in Y_n^0, j<n \}.
\end{cases}\]
\end{enumerate}

The definition above uniquely extends each $\sigma_{\mu}$ to a homomorphism $\overline{\sigma_{\mu}}:[\mathfrak{c}]^{<\omega} \to 2$, which satisfies that, for each $n > 0$, $\xi \in Y_n^0$ and $j<n$, 
\begin{equation}\tag{*}
\overline{\sigma_{\mu}}(\{\xi+j\})= p_{\xi}-\lim_{k \in \omega} \overline{\sigma_{\mu}}(f_{\xi}^j(k)).
\end{equation}

Let now $\overline{\mathcal{A}} \doteq \{\overline{\sigma_{\mu}}: \omega \leq \mu < \mathfrak{c}\}$ and $\tau$ be the weakest (group) topology on $[\mathfrak{c}]^{<\omega}$ making every homomorphism in $\overline{\mathcal{A}}$ continuous. We call this group $G$. We claim that $G$ is Hausdorff. Indeed, given $x \in [\mathfrak{c}]^{<\omega} \setminus \{\emptyset\}$, let $A$ be a suitably closed set containing $x$. We may use Corollary \ref{homfraco} with $E = A$, $I=A \cap Y_0$, $\mathcal{F}=\{x\}$ and, for each $n>0$ and $\xi \in Y_n^0 \cap A$, $d_\xi=(\{\xi\},...,\{\xi+n-1\})$, to fix a homomorphism $\sigma:[A]^{<\omega} \to 2$ so that $\sigma \in \mathcal{A}$ and $\sigma(x)=1$. By construction, there exists $\mu \in [\omega, \mathfrak{c})$ so that $\sigma_{\mu}= \sigma$, and hence $\overline{\sigma_{\mu}}(x)=1$.

\begin{Cla}
For every $n > 0$, $G^n$ is countably pracompact.
\end{Cla}
\begin{proof}[Proof of the claim]\renewcommand{\qedsymbol}{$\blacksquare$}
Fix $n > 0$. We claim that $([X_n]^{<\omega})^n \subset G^n$ is a witness to the countable pracompactness property in $G^n$. Indeed, if $U$ is a nonempty open subset of $G$, we may fix a function $g:S \to 2$, with $S \in [\mathfrak{c}]^{<\omega}$, so that
\[
U \supset \bigcap_{\mu \in S} \overline{\sigma_{\mu}}^{^{-1}}(g(\mu)).
\]
Then, by construction, we may choose $\xi \in X_n^1 \cap (\mu, \mathfrak{c})$ so that $g_{\xi}=g$, and thus $\{\xi\} \in U$, which shows that $[X_n]^{<\omega}$ is dense in $G$, and therefore $([X_n]^{<\omega})^n$ is dense in $G^n$.

We shall now prove that every infinite sequence $\{x_k: k \in \omega\}$ of elements in $([X_n]^{<\omega})^{n}$ has an accumulation point in $G^n$. In fact, by Lemma \ref{bastali}, there are:
\begin{itemize}
\item elements $d_0$,..., $d_{n-1} \in [X_n]^{<\omega}$;
\item a subsequence $(x_{k_l})_{l \in \omega}$;
\item for some $0 \leq t \leq n$, a sequence $(y_{l})_{l \in \omega}$ in $([X_n]^{<\omega})^t$
\item for each $0 \leq s < n$, a function $P_s: t \to 2$,
\end{itemize}
satisfying that
 \begin{enumerate}[i)]
      \item $x_{k_l}^s = \Big( \displaystyle \sum_{j=0}^{t-1}P_s(j)y_{l}^j \Big) \triangle d_s,$
      for every $l \in \omega$ and $0 \leq s<n$.
      \item $\{y_{l}^j: l \in \omega, 0 \leq j<t\}$ is linearly independent.
  \end{enumerate}
  By construction, there exists $\xi \in Y_n^0$ so that $f_{\xi}^j(l) = y_l^j$, for every $l \in \omega$ and $0 \leq j <t$. Since
  \[
  \overline{\sigma_{\mu}}(\{\xi+j\})= p_{\xi}-\lim_{l \in \omega} \overline{\sigma_{\mu}}(f_{\xi}^j(l)),
  \]
for each $\mu \in [\omega, \mathfrak{c})$ and $0 \leq j<n$, we conclude that, for each $0 \leq s<n$, 
\[
\Big( \sum_{j=0}^{t-1}P_s(j)\{\xi+j\} \Big)\triangle d_s = p_{\xi}-\lim_{l \in \omega} x_{k_l}^s,
\]
and therefore $\{x_k: k \in \omega\}$ has an accumulation point in $G^n$. \footnotemark \footnotetext{In fact, the accumulation point obtained even belongs to $([X_n]^{<\omega})^n$ itself. This shows that the subgroup $[X_n]^{<\omega}$ has its nth-power countably compact, for each $n >0$.}
\end{proof}

\begin{Cla}\label{Gomega}
$G^{\omega}$ is not countably pracompact. 
\end{Cla}
\begin{proof}[Proof of the claim]\renewcommand{\qedsymbol}{$\blacksquare$}
Let $Y \subset G^{\omega}$ be a dense subset. Consider the set $\{U_k^j: k \in \omega, j \in \omega\}$ of nonempty open subsets of $G$ given by Lemma \ref{abertosli}. For each $k \in \omega$, we may choose an element $x_k \in Y \cap \prod_{j \leq k} U_k^j \times G^{\omega \setminus k+1}$, and hence 
\[
\{x_k^j: j \in \omega, k \geq j\}
\]
is linearly independent. In what follows, we will show that there exists a subsequence of $\{x_k: k \in \omega\}$ which does not have an accumulation point in $G^{\omega}$.

For an element $D \in [\mathfrak{c}]^{<\omega}$, we define 
\[
\text{SUPP}(D) \doteq \{n > 0: D \cap X_n \neq \emptyset\}.
\]
We will split the proof in two cases.
~\\

\textbf{Case 1:} There exists $j \in \omega$ so that $\bigcup_{k \in \omega} \text{SUPP}(x_{k}^{j})$ is infinite.

In this case, we may fix a subsequence $\{x_{k_m}^j: m \in \omega\}$ such that
\begin{equation}\label{isnot}
\text{SUPP}(x_{k_m}^j)  \setminus \left( \bigcup_{p< m} \text{SUPP}(x_{k_p}^j)\right) \neq \emptyset, \end{equation} for every $m \in \omega$. We may also assume that $k_0 \geq j$, and hence $\{x_{k_m}^{j}: m \in \omega\}$ is linearly independent.
~\\

Now we shall show that, for each $x \in G$, $x$ is not an accumulation point of $\{x_{k_m}^j: m \in \omega\}$. First, note that, given $x \in G$, there exists $N_0 \in \omega$ such that, for every $m \geq N_0$,
\[
\text{SUPP}(x_{k_m}^j) \setminus \left( \bigcup_{p<m} \text{SUPP}(x_{k_p}^j) \cup \text{SUPP}(x)   \right) \neq \emptyset.
\]
In fact, since $\text{SUPP}(x)$ is finite and (\ref{isnot}) holds, there cannot be infinitely many elements $x_{k_m}^j$ such that $\text{SUPP}(x_{k_m}^j) \subset \bigcup_{p<m}\text{SUPP}(x_{k_p}^j) \cup \text{SUPP}(x)$. 

Let
\[
F_0 \doteq \bigcup_{p<N_0} \text{SUPP}(x_{k_p}^j) \cup \text{SUPP}(x)
\]
and, for $i>0$,
\[
F_i \doteq \text{SUPP}(x_{k_{N_0+i-1}}^j) \setminus \left(\bigcup_{p< N_0+i-1} \text{SUPP}(x_{k_p}^j) \cup \text{SUPP}(x)  \right).
\]
Define also, for each $i \in \omega$,
\[
D_i \doteq  \Big( \bigcup_{m \in \omega} x_{k_m}^j \cup x \Big) \cap \Big( \bigcup_{n \in F_i}X_n \Big),
\]
and let $A_i$ be a suitably closed set containing $D_i$ such that $A_i \subset \bigcup_{n \in F_i}X_n$. Since $(F_i)_{i \in \omega}$ is a family of pairwise disjoint sets, we have that $(A_i)_{i \in \omega}$ is also a family of pairwise disjoint sets.

Now we may use Corollary \ref{homfraco} with: $E=A_0$; $I = A_0 \cap Y_0$; $\mathcal{F} = \{x\}$; and, for every $n > 0$ and $\xi \in Y_n^0 \cap A_0$, $d_{\xi} = (\{\xi\},...,\{\xi + n-1\})$, to fix a homomorphism $\theta_0 : [A_0]^{<\omega} \to 2$ such that $\theta_0 \in \mathcal{A}$ and $\theta_0(x)=0$ \footnotemark \footnotetext{If $x = \emptyset$, $\mathcal{F}$ is not linearly independent and thus we cannot use Corollary \ref{homfraco}, but it is clear that we can still find such $\theta_0$.}. For $l>0$, suppose that we have constructed a set of homomorphisms $\{\theta_i: i <l\} \subset \mathcal{A}$ such that
\begin{enumerate}[i)]
   \item $\theta_0(x)=0$.
    \item $\theta_i$ is a homomorphism defined in $\big[ \bigcup_{p \leq i} A_p \big]^{<\omega}$ taking values in $2$, for each $i<l$.
    \item $\theta_i$ extends $\theta_{i-1}$ for each $0<i<l$.
    \item $\theta_{i}(x_{k_{N_0+p}}^j)=1$ for each $0<i<l$ and $p=0,...,i-1$.
\end{enumerate}

Again by Corollary \ref{homfraco}, we may define a homomorphism $\psi_l:[A_l]^{<\omega} \to 2$ so that $\psi_l \in \mathcal{A}$ and 
\[\psi_{l}\Big(x_{k_{N_0+l-1}}^j \setminus \bigcup_{p < l} D_p \Big) + \theta_{l-1}\Big(x_{k_{N_0+l-1}}^j \cap \bigcup_{p < l} D_p \Big) = 1.\]
Now, since $A_l \cap \bigcup_{i < l} A_i = \emptyset$, we may also define a homomorphism $\theta_l : \big[\bigcup_{p \leq l} A_p \big]^{<\omega} \to 2$ extending both $\theta_{l-1}$ and $\psi_l$. By construction, we have that $\theta_l(x)=0$ and $\theta_l(x_{k_{N_0+p}}^j)=1$ for every $p=0,...,l-1$. Also, it follows that $\theta_l \in \mathcal{A}$, since $\psi_l \in \mathcal{A}$ and $\theta_i \in \mathcal{A}$ for every $i<l$. Therefore, there exists a family of homomorphisms $\{\theta_l: l \in \omega\} \subset \mathcal{A}$ satisfying i)-iv) for every $l \in \omega$. 

Letting $A \doteq \bigcup_{i \in \omega} A_i$ and $\theta \doteq \bigcup_{i \in \omega} \theta_i$, the homomorphism $\theta:[A]^{<\omega} \to 2$ satisfies that $\theta \in \mathcal{A}$, since $\theta_i \in \mathcal{A}$ for each $i \in \omega$. Also, $\theta(x)=0$ and $\theta(x_{k_{N_0+p}}^j)=1$ for every $p \in \omega$. By construction, there exists $\mu \in [\omega, \mathfrak{c})$ so that $\theta = \sigma_{\mu}$, thus $\overline{\sigma_\mu}: [\mathfrak{c}]^{<\omega} \to 2$ satisfies that $\overline{\sigma_{\mu}}(x_{k_m}^j)=1$ for each $m \geq N_0$, and $\overline{\sigma_{\mu}}(x)=0$. Hence, the element $x \in G$, which was chosen arbitrarily, is not an accumulation point of $\{x_{k_m}^j: m \in \omega\}$. In particular, $\{x_{k_m}: m \in \omega\}$ does not have an accumulation point in $G^{\omega}$.
~\\~\\
\textbf{Case 2:} For every $j \in \omega$, $M_j \doteq \bigcup_{k \in \omega} \text{SUPP}(x^j_{k})$ is finite.

In this case, we claim that for each $j \in \omega$ there exists a subsequence $\{k_m^j: m \in \omega\}$ so that, for every $i \leq j$ and $n \in M_i$, either the family $\{x_{k_m^j}^i \cap X_n: m \in \omega\}$ is linearly independent or constant. Indeed, for $j=0$ and $n_0 \in M_0$, if there exists an infinite subset of $\{x_k^0 \cap X_{n_0}: k \in \omega\}$ which is linearly independent, we may fix a subsequence $\{k_m^{0,0}:m \in \omega\}$ so that $\{x_{k_m^{0,0}}^{0} \cap X_{n_0}: m \in \omega\}$ is linearly independent; otherwise we may fix a subsequence $\{k_m^{0,0}:m \in \omega\}$ so that $\{x_{k_m^{0,0}}^{0} \cap X_{n_0}: m \in \omega\}$ is constant. Then, if it exists, we may consider another $n_1 \in M_0$ and repeat the process to obtain a subsequence $\{k_m^{0,1}: m \in \omega\}$ which refines $\{k_m^{0,0}:m \in \omega\}$ and satisfies the desired property for $n_0$ and $n_1$. Since $M_0$ is finite, proceeding inductively we may obtain the required subsequence $\{k_m^0: m \in \omega\}$ in the last step. Then, we repeat the process for the next coordinates, always refining the previous subsequence. Now, fix such subsequences $\{k_m^j: m \in \omega\}$, for each $j \in \omega$. We may also suppose that $k_0^j \geq j$ for each $j \in \omega$.

For each $j \in \omega$, let 
\[
\overline{M_j} \doteq \{n \in M_j: \{x_{k_m^j}^j \cap X_n: m \in \omega\} \text{ is linearly independent}\}.
\]
Note that $\overline{M_j} \neq \emptyset$ for every $j \in \omega$, since $\{x_{k_m^j}^{j} \cap X_n: m \in \omega, n \in M_j\}$ generates all the elements in the infinite linearly independent set $\{x_{k_m^j}^{j}: m \in \omega\}$.

Suppose that there exists $j \in \omega$ so that $|\overline{M_j}| > 1$. Fix then $n_0, n_1 \in \overline{M_j}$ distinct. We shall prove that in this case $\{x_{k_m^j}^j: m \in \omega\}$ does not have an accumulation point in $G$.

For that, consider:
\begin{itemize}
    \item $x \in G$ chosen arbitrarily;
     \item $x^0 \doteq x \cap X_{n_0}$, $x^1 \doteq x \cap X_{n_1}$;
    \item $Z_0 \subset X_{n_0}$ a suitably closed set containing $x^0$ and $\bigcup \{x_{k_m^j}^{j} \cap X_{n_0}: m \in \omega\}$, so that $|Z_0 \setminus \bigcup\{x_{k_m^j}^{j} \cap X_{n_0}: m \in \omega\}|= \omega$;
    \item  $Z_1 \subset X_{n_1}$ a suitably closed set containing $x^1$ and $\bigcup \{x_{k_m^j}^{j} \cap X_{n_1}: m \in \omega\}$, so that $|Z_1 \setminus \bigcup\{x_{k_m^j}^{j} \cap X_{n_1}: m \in \omega\}|= \omega$;
    \item $\tilde{E} \doteq Z_0 \cupdot Z_1$;
    \item $I_0 \doteq Z_0 \cap Y_0 (=Z_0 \cap Y_{n_0}^0)$, $I_1 \doteq Z_1 \cap Y_0 (= Z_1 \cap Y_{n_1}^0)$ and $I \doteq I_0 \cupdot I_1$;
    \item for $\xi \in I $, 
 \[
d_{\xi} =
 \begin{cases}
   (\{\xi\},...,\{\xi+n_0-1\}), & \text{if} \ \xi \in I_0\\
   (\{\xi\},...,\{\xi+n_1-1\}), & \text{if} \ \xi \in I_1.
\end{cases}
\]
\end{itemize} 

By Lemma \ref{completo} and Remark \ref{RemLema}, there exists a homomorphism $\tilde{\Phi}:[\tilde{E}]^{<\omega} \to 2$ such that:
\begin{enumerate}[(i)]
\item for every $s \in x^0 \cup x^1$, $\tilde{\Phi}(\{s\})=0$;
\item for every $\xi \in I$, 
 \[
\tilde{\Phi}(\{\xi+j\}) = \begin{cases}
  p_{\xi}-\lim_{k \in \omega}\tilde{\Phi}(f_{\xi}^j(k)), \text{ for every } j<n_0, &\text{if} \ \xi \in I_0\\
  p_{\xi}-\lim_{k \in \omega}\tilde{\Phi}(f_{\xi}^j(k)), \text{ for every } j<n_1, &\text{if} \ \xi \in I_1;
\end{cases}\]

\item $\Big\{m \in \omega: \tilde{\Phi}\Big(x_{k_m^j}^j \cap (X_{n_0} \cup X_{n_1}) \Big) = 0\Big\}$ is finite.
 \end{enumerate}  

Now, fix a suitably closed set $E$ containing $\tilde{E}$, $x$ and $x_{k_m^j}^j$, for each $m \in \omega$, so that $E \cap X_{n_0} = Z_0$ and $E \cap X_{n_1} = Z_1$. Consider the homomorphism $\Phi:[E]^{<\omega} \to 2$ so that, for each $\xi \in E$,
\[
\Phi(\{\xi\}) =
 \begin{cases}
    \tilde{\Phi}(\{\xi\}), & \text{if} \ \xi \in \tilde{E}\\
   0, & \text{if }\xi \notin \tilde{E}.
\end{cases}
\]
In particular, for every $z \in [E]^{<\omega}$ so that $z \cap (X_{n_0} \cup X_{n_1}) = \emptyset$, we have that $\Phi(z)=0$, and for every $z \in [\tilde{E}]^{<\omega}$, $\Phi(z)= \tilde{\Phi}(z)$.

It follows by construction that $\Phi \in \mathcal{A}$. Furthermore, 
\begin{align*}
\Phi(x) &= \Phi\Big( \big(x \cap (X_{n_0} \cup X_{n_1}) \big) \triangle \big(x \setminus (X_{n_0} \cup X_{n_1} ) \big) \Big) \\ &= \Phi\Big( (x \cap X_{n_0}) \triangle (x \cap X_{n_1}) \Big) + \Phi\Big(x \setminus (X_{n_0} \cup X_{n_1} )\Big) = \tilde{\Phi}(x^0) + \tilde{\Phi}(x^1) = 0,
\end{align*}
and, for every $m \in \omega$,
\begin{align*}
\Phi(x_{k_m^j}^j) &= \Phi\Big( \big(x_{k_m^j}^j \cap (X_{n_0} \cup X_{n_1}) \big) \triangle \big(x_{k_m^j}^j \setminus (X_{n_0} \cup X_{n_1} ) \big) \Big) \\ &= \Phi\Big( x_{k_m^j}^j \cap (X_{n_0} \cup X_{n_1}) \Big) + \Phi\Big(x_{k_m^j}^j \setminus (X_{n_0} \cup X_{n_1} )\Big) = \tilde{\Phi}\Big( x_{k_m^j}^j \cap (X_{n_0} \cup X_{n_1}) \Big).
\end{align*}

Thus, 
\[
\Big\{ m \in \omega: \Phi(x_{k_m^j}^j) = \Phi(x)  \Big\}
\]
is finite. Since, by construction, there exists $\mu \in [\omega, \mathfrak{c})$ so that $\Phi = \sigma_{\mu}$, we conclude that $x$ cannot be an accumulation point of $\{x_{k_m^j}^j: m \in \omega\}$. As the element $x \in G$ was chosen arbitrarily, the sequence $\{x_{k_m^j}^j: m \in \omega\}$ does not have an accumulation point in $G$. In particular, $\{x_{k_m^j}: m \in \omega\}$ does not have an accumulation point in $G^{\omega}$.

Therefore, henceforth we may suppose that $|\overline{M_j}|=1$ for every $j \in \omega$. We have two subcases to consider.

\textbf{Case 2.1:} There are $j_0, j_1 \in \omega$ distinct so that $\overline{M_{j_0}} \cap \overline{M_{j_1}} = \emptyset$.

Suppose that $j_1 >j_0$, and let $n_0 \in \overline{M_{j_0}}, n_1 \in \overline{M_{j_1}}$. We shall show that the sequence $\{(x_{k_m^{j_1}}^{j_0},x_{k_m^{j_1}}^{j_1}): m \in \omega\}$ does not have an accumulation point in $G^2$. For this, consider:
\begin{itemize}
    \item $(x^0,x^1) \in G^2$ chosen arbitrarily;
     \item $y^0 \doteq x_{{k_0^{j_1}}}^{j_1} \cap X_{n_{0}}$ and $y^1 \doteq x_{{k_0^{j_1}}}^{j_0} \cap X_{n_{1}}$\footnotemark;\footnotetext{Recall that, by construction, the families $\{x_{{k_m^{j_1}}}^{j_1} \cap X_{n_{0}}: m \in \omega\}$ and $\{x_{{k_m^{j_1}}}^{j_0} \cap X_{n_1}: m \in \omega\}$ are constant.}
    \item $Z_0 \subset X_{n_0}$ a suitably closed set containing $(x^0\cup x^1) \cap X_{n_0}$, $y^0$ and $\bigcup \{x_{k_m^{j_1}}^{j_0} \cap X_{n_0}: m \in \omega\}$, so that $|Z_0 \setminus \bigcup\{x_{k_m^{j_1}}^{j_0} \cap X_{n_0}: m \in \omega\}|= \omega$;
    \item  $Z_1 \subset X_{n_1}$ a suitably closed set containing $(x^0 \cup x^1) \cap X_{n_1}$, $y^1$ and $\bigcup \{x_{k_m^{j_1}}^{j_1} \cap X_{n_1}: m \in \omega\}$, so that $|Z_1 \setminus \bigcup\{x_{k_m^{j_1}}^{j_1} \cap X_{n_1}: m \in \omega\}|= \omega$;
    \item $\tilde{E} \doteq Z_0 \cupdot Z_1$;
    \item $I_0 \doteq Z_0 \cap Y_0 (=Z_0 \cap Y_{n_0}^0)$, $I_1 \doteq Z_1 \cap Y_0 (= Z_1 \cap Y_{n_1}^0)$ and $I \doteq I_0 \cupdot I_1$;
    \item for $\xi \in I $, 
 \[
d_{\xi} =
 \begin{cases}
   (\{\xi\},...,\{\xi+n_0-1\}), & \text{if} \ \xi \in I_0\\
   (\{\xi\},...,\{\xi+n_1-1\}), & \text{if} \ \xi \in I_1.
\end{cases}
\]
\end{itemize} 

By Lemma \ref{completo}, there exists a homomorphism $\tilde{\Phi}:[\tilde{E}]^{<\omega} \to 2$ such that:
\begin{enumerate}[(i)]
\item for every $s \in (x^0 \cup x^1 \cup y^0 \cup y^1) \cap (X_{n_0} \cup X_{n_1})$, $\tilde{\Phi}(\{s\})=0$;
\item for every $\xi \in I$, 
 \[
\tilde{\Phi}(\{\xi+j\}) = \begin{cases}
  p_{\xi}-\lim_{k \in \omega}\tilde{\Phi}(f_{\xi}^j(k)), \text{ for every } j<n_0, &\text{if} \ \xi \in I_0\\
  p_{\xi}-\lim_{k \in \omega}\tilde{\Phi}(f_{\xi}^j(k)), \text{ for every } j<n_1, &\text{if} \ \xi \in I_1;
\end{cases}\]

\item $\Big\{m \in \omega: \Big(\tilde{\Phi}(x_{k_m^{j_1}}^{j_0} \cap X_{n_0}), \tilde{\Phi}(x_{k_m^{j_1}}^{j_1}\cap X_{n_1})\Big) = (0,0)\Big\}$ is finite.
 \end{enumerate}  

Again, fix a suitably closed set $E$ containing $\tilde{E}$, $x^0 \cup x^1$ and $x_{k_m^{j_1}}^{j_0} \cup x_{k_m^{j_1}}^{j_1}$, for each $m \in \omega$, so that $E \cap X_{n_0} = Z_0$ and $E \cap X_{n_1} = Z_1$. Consider the homomorphism $\Phi:[E]^{<\omega} \to 2$ such that, for each $\xi \in E$,
\[
\Phi(\{\xi\}) =
 \begin{cases}
    \tilde{\Phi}(\{\xi\}), & \text{if} \ \xi \in \tilde{E}\\
   0, & \text{if }\xi \notin \tilde{E}.
\end{cases}
\]

It follows by construction that $\Phi \in \mathcal{A}$ and that, for each $i<2$,
\begin{align*}
\Phi(x^i) &= \Phi\Big( \big(x^i \cap (X_{n_0} \cup X_{n_1}) \big) \triangle \big(x^i \setminus (X_{n_0} \cup X_{n_1} ) \big) \Big) \\ &= \Phi\Big( (x^i \cap X_{n_0}) \triangle (x^i \cap X_{n_1}) \Big) + \Phi\Big(x^i \setminus (X_{n_0} \cup X_{n_1} )\Big) = \tilde{\Phi}(x^i \cap X_{n_0})+ \tilde{\Phi}(x^i \cap X_{n_1})=  0.
\end{align*}
Furthermore, for every $m \in \omega$ and $i<2$,
\begin{align*}
\Phi(x_{k_m^{j_1}}^{j_i}) &= \Phi\Big( \big(x_{k_m^{j_1}}^{j_i} \cap (X_{n_0} \cup X_{n_1}) \big) \triangle \big(x_{k_m^{j_1}}^{j_i} \setminus (X_{n_0} \cup X_{n_1} ) \big) \Big) \\ &= \Phi\Big( x_{k_m^{j_1}}^{j_i} \cap (X_{n_0} \cup X_{n_1}) \Big) + \Phi\Big(x_{k_m^{j_1}}^{j_i} \setminus (X_{n_0} \cup X_{n_1} )\Big)\\ &= \tilde{\Phi}\Big( x_{k_m^{j_1}}^{j_i} \cap X_{n_0} \Big) + \tilde{\Phi}\Big( x_{k_m^{j_1}}^{j_i} \cap X_{n_1} \Big) = \tilde{\Phi}\Big( x_{k_m^{j_1}}^{j_i} \cap X_{n_i} \Big).
\end{align*}

Thus, 
\[
\Big\{ m \in \omega: \Big(\Phi(x_{k_m^{j_1}}^{j_0}),\Phi(x_{k_m^{j_1}}^{j_1})) = (\Phi(x^0), \Phi(x^1)\Big)  \Big\}
\]
is finite, and therefore $\{(x_{k_m^{j_1}}^{j_0},x_{k_m^{j_1}}^{j_1}): m \in \omega\}$ does not have an accumulation point in $G^2$. In particular, $\{x_{k_m^{j_1}}: m \in \omega\}$ does not have an accumulation point in $G^{\omega}$.
~\\

\textbf{Case 2.2:} For every $j_0, j_1 \in \omega$, $\overline{M_{j_0}} \cap \overline{M_{j_1}} \neq \emptyset$.

In this case, there exists $n_0>0$ so that $\overline{M_j}=\{n_0\}$ for every $j \in \omega$. To make the notation simpler, from now on we call $\{k_m: m \in \omega\}$ the sequence $\{k_m^{n_0}: m \in \omega\}$. By construction, $\{x_{k_m}^{i}: m \in \omega, i \leq n_0\}$ is linearly independent and, for each $i \leq n_0$, there exists $c_i \in [\mathfrak{c}]^{<\omega}$ so that $c_i \cap X_{n_0}= \emptyset$ and $x_{k_m}^{i} = (x_{k_m}^{i} \cap X_{n_0}) \triangle c_i $, for every $m \in \omega$. Thus, there exists $m_0 \in \omega$ such that
\[
\{x_{k_m}^{i} \cap X_{n_0}: m \geq m_0, \ i \leq n_0 \}
\]
is linearly independent. 

We shall prove that $\{(x_{k_m}^{0},...,x_{k_m}^{n_0}): m \in \omega\}$ does not have an accumulation point in $G^{n_0+1}$. For that, consider:
\begin{itemize}
    \item $x=(x^0,...,x^{n_0}) \in G^{n_0+1}$ chosen arbitrarily;
    \item for each $i=0,...,n_0$, $y_m^i=x_{k_m}^{i} \cap X_{n_0}$ for every $m \geq m_0$.
    \item $\tilde{E} \subset X_{n_0}$ a suitably closed set containing $(x^0 \cup ... \cup x^{n_0}) \cap X_{n_0}$ and $y_m^i$, for every $i \leq n_0$ and $m \geq m_0$, so that $|\tilde{E} \setminus \bigcup\{y_m^i: m \geq m_0, \ i \leq n_0\}|= \omega$;
    \item $I = \tilde{E} \cap Y_0 $ (=$\tilde{E} \cap Y^0_{n_0}$);
    \item for each $\xi \in I$, $d_{\xi}=(\{\xi\},...,\{\xi +n_0-1\})$.
\end{itemize}

By Lemma \ref{caso2.2}, there exists a homomorphism $\tilde{\Phi}:[\tilde{E}]^{<\omega} \to 2$ such that
\begin{enumerate}[(i)]
\item For every $s \in (x^0 \cup ... \cup x^{n_0}) \cap X_{n_0}$, $\tilde{\Phi}(\{s\})=0$.
\item For every $\xi \in I$ and $j<n_0$,
\[
\tilde{\Phi}(\{\xi +j\}) = p_{\xi}-\lim_{k \in \omega} \tilde{\Phi}(f_{\xi}^j(k)).
\]
\item $\{m \geq m_0: (\tilde{\Phi}(y_{m}^{0})...,\tilde{\Phi}(y_{m}^{n_0}))= (0,...,0)\}$ is finite. Note that by construction $\tilde{\Phi}(x^i \cap X_{n_0})=0$ for every $i=0,...,n_0$.
 \end{enumerate}  

Consider $E$ a suitably closed set containing $\tilde{E}$, $x^0 \cup... \cup x^{n_0}$ and $x_{k_m}^{0} \cup ... \cup x_{k_m}^{n_0}$, for each $m \geq m_0$, so that $E \cap X_{n_0}= \tilde{E}$.  Hence, we may define a homomorphism $\Phi:[E]^{<\omega} \to 2$ such that, for each $\xi \in E$,
\[
\Phi(\{\xi\}) =
 \begin{cases}
    \tilde{\Phi}(\{\xi\}), & \text{if} \ \xi \in \tilde{E}\\
   0, & \text{if }\xi \notin \tilde{E}.
\end{cases}
\]
In particular, for every $z \in [E]^{<\omega}$ so that $z \cap X_{n_0} = \emptyset$, we have that $\Phi(z)=0$. Moreover, for every $z \in [\tilde{E}]^{<\omega}$, $\Phi(z)= \tilde{\Phi}(z)$. Similarly to \textbf{Case 2.1}, it follows by construction that $\Phi \in \mathcal{A}$ and that 
\[
\{m \geq m_0: (\Phi(x_{k_m}^{0}),...,\Phi(x_{k_m}^{n_0})) = (\Phi(x^0),...,\Phi(x^{n_0})) \} \text{ is finite.}
\]
Again, we conclude that $x$ cannot be an accumulation point of $\{(x_{k_m}^{0},...,x_{k_m}^{n_0}): m \in \omega\}$.
~\\~\\
\indent Therefore, in any case, we showed that there exists a subsequence of $\{x_{k}: k \in \omega\}$ which does not have an accumulation point in $G^{\omega}$, and thus the group is not countably pracompact.
\end{proof}

\end{proof}

As a corollary of the proof above, we obtain:

\begin{Corol}
Suppose that there are $\mathfrak{c}$ incomparable selective ultrafilters. Then, for each $n \in \omega$, $n>0$, there exists a (Hausdorff) topological group whose nth power is countably compact and the (n+1)th power is not selectively pseudocompact.
\end{Corol}
\begin{proof}
With the same notation of the previous proof, for each $n >0$, we choose the topological subgroup $H \doteq [X_n]^{<\omega} \subset G$. As already mentioned in a footnote, $H^n$ is countably compact. Also, using Lemma \ref{caso2.2} similarly to what was done in \textbf{Case 2.2}, one can show that every sequence $(x_{k}^0,...,x_{k}^{n})_{k \in \omega}$ in $H^{n+1}$ so that $\{x_k^j : k \in \omega , j \leq n\}$ is linearly independent does not have an accumulation point in $H^{n+1}$. Then, it is enough to choose a sequence of nonempty open sets $\{(U_k^0 \times ... \times U_k^n): k \in \omega\} \subset H^{n+1}$, with $\{U_k^j: k \in \omega, j \leq n\}$ as in Lemma \ref{abertosli}, to prove that $H^{n+1}$ is not selectively pseudocompact.
\end{proof}

Recall that in \cite{finpow} the authors proved the same result using CH.

\section{Consistent solutions to the Comfort-like question for countably pracompact groups in the case of infinite successor cardinals}

\begin{Teo}\label{sucessorcardinal}
Suppose that there are $2^\mathfrak{c}$ incomparable selective ultrafilters. Let $\kappa \leq 2^{\mathfrak{c}}$ be an infinite cardinal. Then there exists a (Hausdorff) group $G$ such that $G^{\kappa}$ is countably pracompact and $G^{\kappa^+}$ is not countably pracompact.
\end{Teo}
\begin{proof} 
The required group will be constructed giving a suitable topology to the Boolean group $[2^{\mathfrak{c}}]^{<\omega}$. 

Let $\{X_{\gamma}: \gamma < \kappa\}$ be a partition of $2^{\mathfrak{c}}$ so that $|X_{\gamma}|= 2^{\mathfrak{c}}$ for every $\gamma < \kappa$.
For each $\gamma < \kappa$, we enumerate $X_{\gamma}$ in strictly increasing order as $\{x_{\beta}^{\gamma}: \beta < 2^{\mathfrak{c}}\}$ (in this case, it is clear that, for every $\gamma < \kappa$ and $\beta < 2^\mathfrak{c}$, $\beta \leq x_{\beta}^{\gamma}$). Let also:
\begin{itemize}
    \item $\{J_0, J_1\}$ be a partition of $2^\mathfrak{c}$ so that $|J_0|=|J_1|= 2^\mathfrak{c}$ and that $\omega \subset J_1$;
    \item $X_{\gamma}^0 \doteq \{x_\beta^\gamma: \beta \in J_0\}$ and $X_{\gamma}^1 \doteq \{x_{\beta}^\gamma: \beta \in J_1\}$, for each $\gamma < \kappa$;
    \item $X_0 \doteq \bigcup_{\gamma < \kappa} X_{\gamma}^0$ and $X_1 \doteq \bigcup_{\gamma < \kappa} X_{\gamma}^1$;
    \item $\mathcal{P} \doteq \{p_{\xi}: \xi \in J_0\}$ be a family of incomparable selective ultrafilters, which exists by hypothesis.
\end{itemize}

Now enumerate the set of all injective sequences of $2^\mathfrak{c}$ as $\{I_{\alpha}: \alpha \in J_0\}$, assuming that for every $\alpha \in J_0$, $\text{rng}(I_{\alpha}) \subset \alpha$. Finally, for each $\alpha \in J_0$ and $\gamma < \kappa$, we define the function $f_{\alpha}^{\gamma}: \omega \to [X_{\gamma}]^{<\omega}$ as
\[
f_{\alpha}^{\gamma}(l) = \{x_{I_{\alpha}(l)}^{\gamma}\},
\]
for every $l \in \omega$. Note that, for each $\alpha \in J_0$ and $\gamma < \kappa$, $\text{rng}(f_{\alpha}^{\gamma}) \subset [x_{\alpha}^{\gamma}]^{<\omega}$.

Next we define which are the suitably closed sets of this construction.

\begin{Def}
A set $A \in [2^\mathfrak{c}]^{\omega}$ is suitably closed if, for every $\gamma < \kappa$ and $\beta \in J_0$, if $x^{\gamma}_{\beta} \in A $, then $\bigcup_{l \in \omega}f^{\gamma}_{\beta}(l) \subset A$.
\end{Def}

Let $\mathcal{A} = \{\sigma_{\mu}: \mu \in [\omega, 2^\mathfrak{c})\}$ be an enumeration of all homomorphisms $\sigma:[A]^{<\omega} \to 2$, with $A$ suitably closed, such that
\[
\sigma(\{x_{\beta}^{\gamma}\})= p_{\beta}-\lim_{l \in \omega} \sigma(f^{\gamma}_{\beta}(l)),
\]
for every $\gamma < \kappa$ and $\beta \in J_0$ satisfying that $x_{\beta}^{\gamma} \in A$. We may assume without loss of generality that $\bigcup \text{dom}(\sigma_{\mu}) \subset \mu $, for every $\mu \in [\omega, 2^\mathfrak{c})$.

Next, we will properly extend each homomorphism $\sigma_{\mu}$ in $\mathcal{A}$ to a homomorphism $\overline{\sigma_{\mu}}$ defined in $[2^\mathfrak{c}]^{<\omega}$. For this purpose, we enumerate the set 
\begin{align*}
\mathcal{B} \doteq \{\{(\gamma_0,g_0),...,(\gamma_k,g_k)\}&: k \in \omega, |\{\gamma_0,...,\gamma_k\}|=k+1,  \text{ and, } \\ &\text{ for each } i=0,...,k, \gamma_i <\kappa, \text{ and }g_i:S_i \to 2, \text{ for some } S_i \in [2^\mathfrak{c}]^{<\omega} \}
\end{align*}
as $\{b_{\beta}: \beta \in J_1 \}$. We may also assume that for every $\beta \in J_1 $, $b_{\beta} =  \{(\gamma_0,g_0),...,(\gamma_k,g_k)\}$ is such that $\bigcup_{i=0}^k \text{dom}(g_i) \subset \beta$.

Given $\mu \in [\omega,2^\mathfrak{c})$, if $\xi < 2^\mathfrak{c}$ is such that $\{\xi\} \in \text{dom}(\sigma_{\mu})$, we put $\overline{\sigma_{\mu}}(\{\xi\}) = \sigma_{\mu}(\{\xi\})$. Otherwise, we have a few cases to consider. Firstly, we define the homomorphism in the remaining elements of $X_{\gamma}^1$, for each $\gamma < \kappa$, as described in the next paragraph.

Let $\gamma < \kappa$ and $\xi \in X_\gamma^1$ be so that $\{\xi\} \notin \text{dom}(\sigma_{\mu})$. Let also $\beta \in J_1$ be the element such that $\xi = x^\gamma_{\beta}$. Now,
\begin{itemize}
    \item if there exists a function $g:S \to 2$, $S \in [2^\mathfrak{c}]^{<\omega}$, so that $(\gamma,g) \in b_{\beta}$ and $\mu \in \text{dom}(g)$, we put $\overline{\sigma_{\mu}}(\{\xi\}) = g(\mu)$;
    \item otherwise, we put $\overline{\sigma_{\mu}}(\{\xi\})=0$.
\end{itemize}

Finally, in the remaining elements of $X_\gamma^0$, for each $\gamma < \kappa$, we define $\overline{\sigma_{\mu}}$ recursively, by putting
\[
\overline{\sigma_{\mu}}(\{x_{\beta}^\gamma\}) = p_{\beta}-\lim_{l \in \omega} \overline{\sigma_{\mu}}(f^\gamma_{\beta}(l)),
\]
for each $\beta \in J_0$.

Now we define $\overline{\mathcal{A}} \doteq \{\overline{\sigma_{\mu}}: \mu \in [\omega, 2^\mathfrak{c})\}$. It is clear by the construction that, for each $\mu \in [\omega, 2^\mathfrak{c})$,
\[
\overline{\sigma_{\mu}}(\{x_{\beta}^\gamma\}) = p_{\beta}-\lim_{l \in \omega} \overline{\sigma_{\mu}}(f^\gamma_{\beta}(l)),
\]
for every $\gamma < \kappa$ and $\beta \in J_0$. Let $G$ be the group $[2^\mathfrak{c}]^{<\omega}$ endowed with the topology generated by the homomorphisms in $\overline{\mathcal{A}}$. 

Given $x \in G$, we define, similarly as before,
\[
\text{SUPP}(x) = \{\gamma < \kappa: x \cap X_\gamma \neq \emptyset\}.
\]

We claim that $G$ is Hausdorff. Indeed, let $x \in [2^\mathfrak{c}]^{<\omega} \setminus \{\emptyset\}$ and, given $\gamma \in \text{SUPP}(x)$, $z = x \cap X_\gamma$. Let also $A_0 \subset X_\gamma$ be a suitably closed set containing $z$. In order to use Corollary \ref{homfraco}, consider:
\begin{itemize}
\item $E = A_0$;
\item $I=A_0 \cap X_\gamma^0$;
\item $\mathcal{F}=\{z\}$;
\item $\{q_\xi: \xi \in I\} \subset \mathcal{P}$ so that, for each $\xi \doteq x_{\beta}^\gamma \in I$, $q_{\xi} \doteq p_\beta$;
\end{itemize}
and, for each $\xi \doteq x_\beta^\gamma \in I$,
\begin{itemize}
\item $k_\xi =1$;
\item $g_\xi=f_\beta^\gamma$;
\item $d_\xi=\{x_ \beta^\gamma\}$.
\end{itemize}
By Corollary \ref{homfraco}, we may fix a homomorphism $\sigma_0:[A_0]^{<\omega} \to 2$ so that $\sigma_0 \in \mathcal{A}$ and $\sigma_0(z)=1$. Now, let $A$ be a suitably closed set containing $x$ so that $A \cap X_\gamma = A_0$ and $\sigma:[A]^{<\omega} \to 2$ be a homomorphism so that
 \[
\sigma(\{\xi\}) =
 \begin{cases}
    \sigma_0(\{\xi\}), & \text{if} \ \xi \in A_0\\
   0, & \text{if }\xi \notin A_0.
\end{cases}
\]
Then, $\sigma \in \mathcal{A}$ and $\sigma(x)=\sigma\big((x\cap X_\gamma) \triangle (x \setminus X_\gamma)\big)=1$. By construction, there exists $\mu \in [\omega,2^\mathfrak{c})$ so that $\sigma_\mu=\sigma$, and hence $\overline{\sigma_{\mu}}(x)=1$.

\begin{Cla}
$G^{\kappa}$ is countably pracompact.
\end{Cla}
\begin{proof}[Proof of the claim]\renewcommand{\qedsymbol}{$\blacksquare$}
We claim that $\{(\{x_{\beta}^{\gamma}\})_{\gamma < \kappa}: \beta < 2^\mathfrak{c}\} \subset G^{\kappa}$ is a dense subset for which every infinite subset has an accumulation point in $G^{\kappa}$.

For $k \in \omega$, let $\{\gamma_0,...,\gamma_k\} \subset \kappa$ be a finite set of size $k+1$ and, for each $i \in  \{0,...,k\}$, let 
\begin{itemize}
\item  $\mu^i_0,...,\mu^i_{j_i} \in [\omega,2^{\mathfrak{c}})$, for some $j_i \in \omega$;
\item  $g_i:\{\mu_0^i,...,\mu_{j_i}^i\} \to 2$ be a function.
\end{itemize}
We shall prove that, if 
\[
\bigcap_{p=0}^{j_i} (\overline{\sigma_{\mu_{p}^i})}^{-1}(g_i(\mu^i_p)) \neq \emptyset,
\]
for every $i=0,...,k$, then there exists $\beta_0 \in J_1$ so that $\{x_{\beta_0}^{\gamma_i}\} \in \bigcap_{p=0}^{j_i} (\overline{\sigma_{\mu_{p}^i}})^{-1}(g_i(\mu_p^i))$ for each $i = 0,...,k$. For that, let $\beta_0 \in J_1$ be so that
\[
\{(\gamma_0,g_0),...,(\gamma_k,g_k)\} = b_{\beta_0}.
\]
 Since, by construction, $\bigcup \text{dom}(\sigma_\mu) \subset \mu$ for every $\mu \in[\omega, 2^\mathfrak{c})$, $\bigcup_{i=0}^k \text{dom}(g_i) \subset \beta_0$ and $\beta_0 \leq x_{\beta_0}^\gamma$ for every $\gamma < \kappa$, it follows that $\overline{\sigma_{\mu_p^i}}(\{x_{\beta_0}^{\gamma_i}\}) = g_i(\mu_p^i)$ for each $i=0,...,k$ and $p = 0,... , j_i$, as we wanted.

Furthermore, given an injective sequence $I_{\alpha}: \omega \to 2^\mathfrak{c}$, for some $\alpha \in J_0$, we claim that $\{(\{x_{I_{\alpha}(l)}^\gamma\})_{\gamma < \kappa}: l \in \omega\}$ has $(\{x_{\alpha}^\gamma\})_{\gamma < \kappa}$ as accumulation point. Indeed, for every $\mu \in [\omega, 2^\mathfrak{c})$, by construction,
\[
\overline{\sigma_{\mu}}(\{x_{\alpha}^\gamma\}) = p_{\alpha}-\lim_{l \in \omega} \overline{\sigma_{\mu}}(\{x_{I_{\alpha}(l)}^\gamma\}),
\]
for each $\gamma < \kappa$. 
\end{proof}

\begin{Cla}
$G^{\kappa^+}$ is not countably pracompact.
\end{Cla}
\begin{proof}[Proof of the claim]\renewcommand{\qedsymbol}{$\blacksquare$}
Since the proof of this claim is similar to the proof of \textbf{Claim 2} of Theorem \ref{casoomega}, we omit the details of some arguments. 

Let $Z \subset G^{\kappa^+}$ be a dense subset. We shall show that there exists a sequence in $Z$ that does not have an accumulation point in $G^{\kappa^+}$. We will again split the proof of this claim in two cases.
~\\

\textbf{Case 1:}  There exists $j \in \kappa^+$ so that $\bigcup_{z \in Z} \text{SUPP}(z^j)$ is infinite.
~\\

In this case, we may fix a sequence $\{z_m: m \in \omega\} \subset Z$ so that
\[
\text{SUPP}(z^j_m) \setminus \bigcup_{p < m} \text{SUPP}(z^j_p) \neq \emptyset,
\]
for every $m \in \omega$. We shall show that, for a given $y \in G$, $y$ is not an accumulation point of $\{z^j_m: m \in \omega\}$. In particular, this shows that $\{z_m : m \in \omega\}$ does not have an accumulation point in $G^{\kappa^+}$.

\begin{Rem}
Although the arguments are analogous to those in \textbf{Case 1} of \textbf{Claim} \ref{Gomega}, as we are about to see, there is another technical complication in this case. We can only guarantee the validity of Corollary \ref{homfraco} for suitably closed sets $A$ so that $A\subset X_{\gamma}$, for some $\gamma < \kappa$. In fact, while the mapping $\xi \in X_0 \cap A \to q_{\xi} \in \mathcal{P}$ has to be injective\footnotemark \footnotetext{Indeed, $X_0 \cap A $ will be the set $I$, in the notation of Corollary \ref{homfraco}, and then $(q_{\xi})_{\xi \in I}$ has to be a family of incomparable selective ultrafilters.}, we wish to map $x_{\beta}^\gamma$, for $\beta \in J_0$ and $\gamma < \kappa$, to $p_{\beta}$.
\end{Rem}
\vspace{0.8 cm}

Let:
\begin{enumerate}[1)]
    \item $M_0 \in \omega$ be such that, for every $m \geq M_0$,
\[
\text{SUPP}(z_{m}^j) \setminus \left( \bigcup_{p<m} \text{SUPP}(z_{p}^j) \cup \text{SUPP}(y)   \right) \neq \emptyset;
\]
    
    \item $ \displaystyle
F_0 \doteq \bigcup_{p<M_0} \text{SUPP}(z_{p}^j) \cup \text{SUPP}(y);$
    
    \item for each $i>0$,
\[
F_i \doteq \text{SUPP}(z_{M_0+i-1}^j) \setminus \left(\bigcup_{p< M_0+i-1} \text{SUPP}(z_{p}^j) \cup \text{SUPP}(y)  \right); \]

\item for each $i \in \omega$,
\[
D_i \doteq \Big(\bigcup_{m \in \omega} z_{m}^j \cup y  \Big) \cap \Big( \bigcup_{\gamma \in F_i}X_{\gamma} \Big);
\]

\item $A_i$ be a suitably closed set containing $D_i$ such that $A_i \subset \bigcup_{\gamma \in F_i} X_{\gamma}$;

\item for each $i \in \omega$, $\gamma_i \in F_i$ be arbitrarily chosen;

\item for each $i \in \omega$, $A_i^{0} \subset X_{\gamma_i}$ be a suitably closed set containing $A_i \cap X_{\gamma_i}$.
    
\end{enumerate}

In order to use Corollary \ref{homfraco}, consider
\begin{itemize}
    \item $E= A_0^0$;
    \item  $I= A_0^0 \cap X_{\gamma_0}^0$;
    \item $\{q_{\xi}: \xi \in I\} \subset \mathcal{P}$ so that, for each $\xi \doteq x_{\beta}^{\gamma_0} \in I$, $q_{\xi} \doteq p_{\beta}$;
\end{itemize}
and, for each $\xi \doteq x^{\gamma_0}_{\beta} \in I$,
\begin{itemize}
    \item $k_{\xi}=1$;
    \item $g_{\xi} = f_{\beta}^{\gamma_0}$;
    \item  $d_{\xi}=\{x^{\gamma_0}_{\beta}\}$.
    \end{itemize}
    By Corollary \ref{homfraco}, we may ensure the existence of a homomorphism $\tilde{\theta}_0:[A_0^0]^{<\omega} \to 2$ such that $\tilde{\theta}_0 \in \mathcal{A}$ and $\tilde{\theta}_0(y \cap X_{\gamma_0})=0$. Then, we define $\theta_0:[A_0]^{<\omega} \to 2$ so that, for every $\xi \in A_0$,
    \[
\theta_0(\{\xi\}) =
 \begin{cases}
    \tilde{\theta}_0(\{\xi\}), & \text{if} \ \xi \in A_0^0\\
   0, & \text{if }\xi \notin A_0^0.
\end{cases}
\]
    Note that, in this case, we still have $\theta_0 \in \mathcal{A}$, and also 
    \[
    \theta_0(y) = \theta_0(y \cap X_{\gamma_0}) + \theta_0(y \setminus X_{\gamma_0}) = \tilde{\theta_0}(y \cap X_{\gamma_0}) = 0.
    \]

Suppose that we have constructed a set of homomorphisms $\{\theta_i: i <l\} \subset \mathcal{A}$, for $l>0$, such that:
\begin{enumerate}[i)]
\item $\theta_i$ is a homomorphism defined in $\big[ \bigcup_{p \leq i} A_p \big]^{<\omega}$ taking values in $2$, for each $i<l$;
   \item $\theta_0(y)=0$;
    \item $\theta_i$ extends $\theta_{i-1}$ for each $0<i<l$;
    \item $\theta_{i}(z_{M_0+p}^j)=1$ for each $0<i<l$ and $p=0,...,i-1$.
\end{enumerate}

Again, in order to use Corollary \ref{homfraco}, consider:
\begin{itemize}
\item $E= A_l^0$;
\item $I= A_l^0 \cap X_{\gamma_l}^0$;
\item $\{q_{\xi}: \xi \in I\} \subset \mathcal{P}$ so that, for each $\xi \doteq x_{\beta}^{\gamma_l} \in I$, $q_{\xi} \doteq p_{\beta}$;
\end{itemize}
and, for each $\xi \doteq x^{\gamma_l}_{\beta} \in I$,
\begin{itemize}
    \item $k_{\xi}=1$;
    \item $g_{\xi} = f_{\beta}^{\gamma_l}$;
    \item  $d_{\xi}=\{x^{\gamma_l}_{\beta}\}$.
    \end{itemize}

By Corollary \ref{homfraco}, we may ensure the existence of a homomorphism $\tilde{\psi}:[A_l^0]^{<\omega} \to 2$ so that $\tilde{\psi} \in \mathcal{A}$ and 
\[\tilde{\psi}\Big(z_{M_0+l-1}^j \cap X_{\gamma_l}\Big) + \theta_{l-1}\Big(z_{M_0+l-1}^j \setminus \bigcup_{\gamma \in F_l} X_{\gamma} \Big) = 1.\]

Then, we define $\psi: [A_l]^{<\omega} \to 2$ so that, for every $\xi \in A_l$,
 \[
\psi(\{\xi\}) =
 \begin{cases}
    \tilde{\psi}(\{\xi\}), & \text{if} \ \xi \in A_l^0\\
   0, & \text{if }\xi \notin A_l^0.
\end{cases}
\]

Let $\theta_l : \big[\bigcup_{p \leq l} A_p \big]^{<\omega} \to 2$ be a homomorphism extending both $\theta_{l-1}$ and $\psi$. By construction, we have that $\theta_l(y)=0$, $\theta_l \in \mathcal{A}$, and also that
\begin{align*} 
\theta_l(z_{M_0+l-1}^j) &= \theta_l\Big(z_{M_0+l-1}^j \cap \bigcup_{\gamma \in F_l} X_{\gamma} \Big) + \theta_{l}\Big( z_{M_0+l-1}^j \setminus \bigcup_{\gamma \in F_l }X_{\gamma} \Big)  \\ &= \tilde{\psi}\Big(z_{M_0+l-1}^j \cap X_{\gamma_l} \Big) + \psi\Big( z_{M_0+l-1}^j \cap \bigcup_{\gamma \in F_l \setminus \{\gamma_l\}}X_{\gamma} \Big) + \theta_{l-1}\Big( z_{M_0+l-1}^j \setminus \bigcup_{\gamma \in F_l }X_{\gamma} \Big) \\ &= \tilde{\psi}\Big(z_{M_0+l-1}^j \cap X_{\gamma_l} \Big) + \theta_{l-1}\Big( z_{M_0+l-1}^j \setminus \bigcup_{\gamma \in F_l }X_{\gamma} \Big) =1 .
\end{align*}

Moreover, it follows by construction that $\theta_l(z_{M_0+p}^j) = \theta_{l-1}(z_{M_0+p}^j)=1 $ for each $0\leq p<l-1$. Therefore, there exists a family of homomorphisms $\{\theta_i: i \in \omega\} \subset \mathcal{A}$ satisfying i)-iv) for every $l \in \omega$. 

Letting $A \doteq \bigcup_{i \in \omega} A_i$, the homomorphism $\theta \doteq \bigcup_{i \in \omega} \theta_i : [A]^{<\omega} \to 2$, satisfies that:
\begin{itemize}
    \item $\theta \in \mathcal{A}$;
    \item $\theta(y)=0$;
    \item $\theta(z_{{M_0+p}}^j)=1$ for every $p \in \omega$.
\end{itemize}

By construction, there exists $\mu \in [\omega,2^\mathfrak{c})$ so that $\theta = \sigma_{\mu}$, thus $\overline{\sigma_\mu}: [2^\mathfrak{c}]^{<\omega} \to 2$ satisfies that $\overline{\sigma_{\mu}}(z_m^j)=1$ for each $m \geq M_0$, and $\overline{\sigma_{\mu}}(y)=0$. Hence, $y \in G$ is not an accumulation point of $\{z_{m}^j: m \in \omega\}$. 
~\\

\textbf{Case 2:} For every $j \in \kappa^+$, $M_j \doteq \bigcup_{z \in Z} \text{SUPP}(z^j)$ is finite.
~\\

Since, in this case,
\[\bigcup_{F \in [\kappa]^{<\omega}} \Big\{j \in \kappa^+: M_j = F \Big\} = \kappa^+,\] 
there exists $F_0 \in [\kappa]^{<\omega}$ so that $N \doteq \Big\{j \in \kappa^+: M_j= F_0 \Big\}$ is infinite. Choose $N_0 \subset N$ so that $|N_0|=\omega$, and let $\{j_i: i \in \omega\}$ be an enumeration of $N_0$. 

Now, consider the set $\{U_k^i: k \in \omega, i \in \omega\}$ of nonempty open subsets of $G$ given by Lemma \ref{abertosli}. For each $k \in \omega$, we may choose an element $z_k \in Z \cap \prod_{i \leq k} U_k^i \times G^{\kappa^+ \setminus \{j_0,...,j_k\}}$. Similarly to what was done in \textbf{Case 2} of \textbf{Claim} \ref{Gomega}, we can fix 
 a subsequence $\{k_m^i: m \in \omega\}$, for each $i \in \omega$, so that: 
\begin{itemize}
\item $\{k_m^{i+1}: m \in \omega\}$ refines $\{k_m^i: m \in \omega\}$, for each $i \in \omega$;
\item for every $i \in \omega$, $p \leq i$ and $\gamma \in F_0$, either the family $\{z_{k_m^i}^{j_p} \cap X_\gamma: m \in \omega\}$ is linearly independent or constant;
\item $k_0^i \geq i$, for each $i \in \omega$.
\end{itemize}
Notice at this point that
\[
\Big\{ z_{k_m^i}^{j_i}: i \in \omega, m \in \omega \Big\}
\]
is linearly independent. For each $i \in \omega$, let 
\[
\overline{M_{j_i}} \doteq \Big\{\gamma \in F_0: \{z_{k_m^i}^{j_i} \cap X_\gamma: m \in \omega\} \text{ is linearly independent} \Big\}.
\]
Again, we have that $\overline{M_{j_i}} \neq \emptyset$ for every $i \in \omega$. Then, choose $a, b \in \omega$, $b>a$, so that $M \doteq \overline{M_{j_a}}= \overline{M_{j_b}}$. In this case, there exist $c_a, c_b \in [2^\mathfrak{c}]^{<\omega}$ so that 
\[
z_{k_m^b}^{j_l} = \Big( z_{k_m^b}^{j_l} \cap \bigcup_{\gamma \in M} X_\gamma \Big) \triangle c_l,
\]
for each $l \in \{a,b\}$ and $m \in \omega$. Thus, there exists $m_0 \in \omega$ so that
\[
\Big\{ z_{k_m^b}^{j_l} \cap \bigcup_{\gamma \in M} X_\gamma: m \geq m_0, \ l \in \{a,b\}  \Big\}
\]
is linearly independent. By Lemma \ref{caso2teo2}, we may fix a subsequence $\{k_m: m \in \omega\}$ of $\{k_m^b: m \in \omega\}$ and $\gamma_0 \in M$ so that
\[
\Big\{z_{k_m}^{j_l} \cap X_{\gamma_0}: m \in \omega, \ l \in \{a,b\} \Big\}
\]
is linearly independent.

We shall show that $\{(z_{k_m}^{j_a}, z_{k_m}^{j_b}): m \in \omega\}$ does not have an accumulation point in $G^2$. For this purpose, consider:
\begin{itemize}
    \item $x=(x^0,x^1) \in G^{2}$ chosen arbitrarily;
    \item for each $l \in \{a,b\}$ and $m \in \omega$, $y_m^l \doteq z_{k_m}^{j_l} \cap X_{\gamma_0}$;
    \item $\tilde{E} \subset X_{\gamma_0}$ a suitably closed set containing $(x^0 \cup x^1) \cap X_{\gamma_0}$ and $y_m^l$, for each $l \in \{a,b\}$ and $m \in \omega$, so that $|\tilde{E} \setminus \bigcup \{y_m^l: l \in \{a,b\}, \ m \in \omega\}|= \omega$;
    \item $I = \tilde{E} \cap X_{\gamma_0}^0$;
    \item $\{q_{\xi}: \xi \in I\} \subset \mathcal{P}$ so that, for each $\xi \doteq x_{\beta}^{\gamma_0} \in I$, $q_{\xi}=p_{\beta}$;
    \item for each $\xi \doteq x_{\beta}^{\gamma_0} \in I$, $d_{\xi}=\{x_{\beta}^{\gamma_0}\}$;
    \item  for each $\xi \doteq x_{\beta}^{\gamma_0} \in I$, $g_{\xi} = f_{\beta}^{\gamma_0}$.
\end{itemize}

By Lemma \ref{caso2.2}, there exists a homomorphism $\tilde{\Phi}:[\tilde{E}]^{<\omega} \to 2$ so that:
\begin{enumerate}[i)]
    \item for every $s \in (x^0 \cup x^1) \cap X_{\gamma_0}$, $\tilde{\Phi}(\{s\})=0$;
    \item for every $\xi = x_{\beta}^{\gamma_0} \in I$,
    \[
    \tilde{\Phi}(\{x_{\beta}^{\gamma_0}\}) = p_{\beta}-\lim_{l \in \omega} \tilde{\Phi}(f_{\beta}^{\gamma_0}(l));
    \]
    \item $\{m \in \omega: (\tilde{\Phi}(y_m^a),\tilde{\Phi}(y_m^b))=(0,0)\}$ is finite. 
\end{enumerate}

Now, we may consider $E$ a suitably closed set containing $\tilde{E}$, $x^0 \cup x^1$, and $z_{k_m}^{j_l}$, for each $l \in \{a,b\}$ and $m \in \omega$, so that $E \cap X_{\gamma_0}= \tilde{E}$. Let $\Phi:[E]^{<\omega} \to 2$ be the homomorphism such that, for each $\xi \in E$,
\[
\Phi(\{\xi\}) =
 \begin{cases}
    \tilde{\Phi}(\{\xi\}), & \text{if} \ \xi \in \tilde{E}\\
   0, & \text{if }\xi \notin \tilde{E}.
\end{cases}
\]
Then, $\Phi \in \mathcal{A}$, 
\[
\Phi(x^0) = \Phi(x^1)=0,
\]
and, for each $l \in \{a,b\}$ and $m \in \omega$,
\[
\Phi(z_{k_m}^{j_l}) = \tilde{\Phi}(y^l_m) + \Phi(z_{k_m}^{j_l} \setminus X_{\gamma_0}) = \tilde{\Phi}(y^l_m).
\]
Thus, we conclude that 
\[
\Big\{m \in \omega: (\Phi(z_{k_m}^{j_a}), \Phi(z_{k_m}^{j_b})) = (\Phi(x^0), \Phi(x^1))  \Big\}
\]
is finite. Since, by construction, there exists $\mu \in [\omega, 2^{\mathfrak{c}})$ so that $\sigma_{\mu}= \Phi$, we conclude that $x$ cannot be an accumulation point of $\{(z_{k_m}^{j_a}, z_{k_m}^{j_b}): m \in \omega\}$. Since $x \in G^2$ is arbitrary, $\{z_{k_m}: m \in \omega\} \subset Z$ does not have an accumulation point in $G^{\kappa^+}$. 

Therefore, $G^{\kappa^+}$ is not countably pracompact.

\end{proof}

\end{proof}

\section{Some remarks and questions}

As we already mentioned in the first section, in the case of selectively pseudocompact groups, the Comfort-like Question \ref{comsel} is still not solved consistently only for the case $\alpha = \omega$:

\begin{Ques}
Is there a topological group $G$ so that $G^{k}$ is selectively pseudocompact for every $k \in \omega$, but $G^{\omega}$ is not selectively pseudocompact?
\end{Ques}

In the case there is a positive consistent answer to the  above question, one can also ask:

\begin{Ques}
Is there a topological group $G$ so that $G^{k}$ is countably compact for every $k \in \omega$, but $G^{\omega}$ is not selectively pseudocompact?
\end{Ques}

Regarding countably pracompact topological groups, Theorem \ref{sucessorcardinal} for $\kappa = 2^{\mathfrak{c}}$ shows that there exists a group $G$ so that $G^{2^{\mathfrak{c}}}$ is countably pracompact but $G^{(2^{\mathfrak{c}})^{+}}$ is not countably pracompact. Interestingly, for countably compact spaces we know that this is not the case: given a Hausdorff topological space $X$, if $X^{2^{\mathfrak{c}}}$ is countably compact, then $X^{\alpha}$ is countably compact for every $\alpha > 2^{\mathfrak{c}}$. Thus, it may be interesting to study the following questions further. The first one is a stronger version of Question \ref{compra} for $\alpha= \omega$, which we solved in this paper.

\begin{Ques}
Is there a topological group $G$ so that $G^{k}$ is countably compact for every $k \in \omega$ and $G^{\omega}$ is not countably pracompact?
\end{Ques}

\begin{Ques}
For which limit cardinals $\omega < \alpha \leq 2^{\mathfrak{c}}$ is there a topological group $G$ such that $G^{\gamma}$ is countably pracompact for every cardinal $\gamma < \alpha$, but $G^\alpha$ is not countably pracompact?
\end{Ques}

\begin{Ques}
For which cardinals $\alpha > (2^{\mathfrak{c}})^{+}$ is there a topological group $G$ such that $G^{\gamma}$ is countably pracompact for all cardinals $\gamma < \alpha$, but $G^{\alpha}$ is not countably pracompact?
\end{Ques}

Also, it is natural to ask which the \textit{stopping point} is, if any:

\begin{Ques}
Is there a cardinal $\kappa$ such that, for each topological group $G$, $G^\kappa$ countably pracompact implies that $G^\gamma$ is countably pracompact for every $\gamma > \kappa$?
\end{Ques}

In ZFC, as mentioned, we do not even know answers to the following questions.

\begin{Ques}[ZFC]
\begin{enumerate}[a)] 

    \item Is there a selectively pseudocompact group whose square is not selectively pseudocompact?
    \item \textbf{(stronger version)} Is there a countably compact group whose square is not selectively pseudocompact?
\end{enumerate}
\end{Ques}

\begin{Ques}[ZFC]
\begin{enumerate}[a)]

\item Is there a countably pracompact group whose square is not countably pracompact?
\item \textbf{(stronger version)} Is there a countably compact group whose square is not countably pracompact?
\end{enumerate}
\end{Ques}

\bibliographystyle{plain}
\bibliography{main}
\Addresses

\end{document}